\documentclass[12pt,letterpaper]{article}
\usepackage[hmargin={25mm,25mm},vmargin={25mm,25mm}]{geometry}
\usepackage{amsmath,amsthm,amssymb,bbm}
\usepackage{graphicx}
\usepackage{fancyhdr}
\usepackage{hyperref}
\usepackage{lipsum}
\usepackage[center]{caption}
\usepackage{verbatim, ifpdf}
\usepackage{graphicx, color}
\usepackage{enumerate}
\usepackage{csvsimple}
\usepackage[parfill]{parskip}
\usepackage [english]{babel}
\usepackage [autostyle, english = american]{csquotes}
\usepackage{centernot}
\usepackage{mathtools}
\usepackage{ stmaryrd }
\usepackage[backend=biber]{biblatex}
\usepackage{setspace}
\usepackage{fullpage}
\usepackage{colonequals}

\MakeOuterQuote{"}
\DeclareMathOperator{\End}{End}
\DeclareMathOperator{\can}{can}
\title{$p$-adic Equidistribution of Special Loci in a Product of Modular Curves}
\author{Dan Townsend}

\usepackage{microtype}
\begin{document}
	
	\newtheoremstyle{GapBefore}
	{1em}
	{0.2em}
	{\itshape}
	{}
	{\bfseries}
	{.}
	{.5em}
	{}
	
	\theoremstyle{GapBefore}
	\newtheorem{thm}{Theorem}[section]
	\newtheorem{lemma}[thm]{Lemma}
	\newtheorem{prop}[thm]{Proposition}
	\newtheorem{cor}[thm]{Corollary}
	\newtheorem{conj}[thm]{Conjecture}
	\theoremstyle{remark}
	\newtheorem{rem}[thm]{Remark}
	\newtheorem{defn}[thm]{Definition}
	\newtheorem{calc}[thm]{Calculation}
	\newcommand{\arr}{\ensuremath{\mathbb{R}}}
	\newcommand{\ay}{\ensuremath{\mathbb{A}}}
	\newcommand{\see}{\ensuremath{\mathbb{C}}}
	\newcommand{\kew}{\ensuremath{\mathbb{Q}}}
	\newcommand{\pee}{\ensuremath{\mathbb{P}}}
	\newcommand{\zed}{\ensuremath{\mathbb{Z}}}
	\newcommand{\enn}{\ensuremath{\mathbb{N}}}
	\newcommand{\emm}{\ensuremath{\mathfrak{M}}}
	\newcommand{\ok}{\ensuremath{\mathcal{O}_K}}
	\newcommand{\oh}{\ensuremath{\mathcal{O}}}
	\newcommand{\ol}{\ensuremath{\mathcal{O}_L}}
	\newcommand{\eff}{\ensuremath{\mathbb{F}}}
	\newcommand{\rhobar}{\ensuremath{\bar{\rho}}}
	\newcommand{\gee}{\ensuremath{\mathcal{G}}}
	\newcommand{\fpbar}{\ensuremath{\overline{\mathbb{F}}_p}}
	\newcommand{\cpun}{\ensuremath{\overline{\mathbb{C}_p^\text{un}}}}
	\newcommand{\cp}{\ensuremath{\mathbb{C}_p}}
	\newcommand{\Sha}{\textls{III}}
	\newcommand{\tdlc}{totally disconnected locally compact group}

	\maketitle
	
	\begin{abstract}
		Let $C\subset X(1)\times X(1)$ be a smooth curve defined over $\cp$ with irreducible reduction. We study the intersection loci of $C$ with the modular subvarieties $Y_0(n)$. If the curve $C$ avoids points where both coordinates have supersingular reduction, or if a sequence $Y_0(a_n)$ is taken where the numbers $a_n$ have increasing divisibility by $p$, then these loci equidistribute to the unique canonical point of the analytification, $C^\text{Berk}$. If neither condition is satisfied, we expect equidistribution to fail and give a family of examples whose behaviour is believed to be typical. We also study the accumulation points of the set $C\cap\bigcup_nY_0(n)$ and show that this set is always non-discrete.
	\end{abstract}

	\section{Introduction}

	Let $p$ be a prime number and view $Y(1)_{j_1}\times Y(1)_{j_2}\cong \cp^2$ as the product of $j$-lines over \cp. This is the moduli space of pairs of elliptic curves over \cp. This space contains, for each $n\geq1$, an copy of the modular subvariety $Y_0(n)$, coming from the map which sends a cyclic $n$-isogeny $(E_1\xrightarrow{n}E_2)$ to the point $(j(E_1),j(E_2))$. 
	
	Let $C$ be an irreducible curve in the compactification $X(1)\times X(1)\cong \pee^1_{\cp}\times\pee^1_{\cp}$ of this product of $j$-lines, which is not a Shimura subvariety (here, this means that $C$ is not $Y_0(n)$ or of the form $\{j_1\}\times \pee^1$ or $\pee^1\times \{j_2\}$ for $j_1$ or $j_2$ a $j$-invariant corresponding to an elliptic curve with complex multiplication (CM)). Consider the intersections of $C$ with modular curves. Let $\bar{\delta}_n(C)$ be the probability measure supported on the intersection of $C$ with $Y_0(n)$ (counting with multiplicity, see \S\ref{Notation} for a precise definition). The main result is: 
	
	\begin{thm} \label{Main}
		Let $C$ be a smooth curve in $X(1)_{j_1}\times X(1)_{j_2}$, as above, defined over \cp \ and let $\mathcal{C}$ be a semistable model of $C$. Suppose that the reduction $\mathcal{C}_{\fpbar}$ is irreducible.
		\begin{enumerate}[(i)]
			\item Let $a_n$ be a sequence of positive integers such that $v_p(a_n)\to \infty$. Then there is weak convergence of measures \[\bar{\delta}_{a_n}(C)\to \delta_{x_\text{can}}.\]
			\item Suppose further that $C$ avoids points both of whose reductions are supersingular. Then there is weak convergence of measures
			\[\bar{\delta}_{n}(C)\to \delta_{x_\text{can}}.\]
		\end{enumerate}
	\end{thm}
	
	This is not the first equidistribution result of this flavour: in the 1-dimensional setting, Sebasti\'an Herrero, Ricardo Menares and Juan Rivera-Letelier analysed which sequences of CM points were equidistributed in the $j$-line over \cp. All sequences of ordinary discriminants had the same limiting measure (Theorem A in \cite{HMR20}), as did sequences of supersingular discriminants becoming increasingly divisible by $p$ (that is, with a condition analogous to that appearing in (i) of Theorem \ref{Main}). On the other hand, ranging over sequences of supersingular discriminants with bounded $p$-divisibility gave infinitely many possible limiting measures (Corollary 1.2 in \cite{HMR21}). This is in contrast with the answer to the same question over the complex numbers, where all sequences with discriminant tending to $-\infty$ equidistribute to the same measure. This question about equidistribution of CM $j$-invariants over the complex numbers was studied over many decades, finally being resolved in general in \S2.3 of \cite{CU04}. The complex analogue of Theorem \ref{Main} here is contained in work of Salim Tayou (\cite{Tay20}, Theorem 1.1) which proves the result for any GSpin Shimura variety.

	The number of special points on $C$ (that is, those where both coordinates correspond to elliptic curves with complex multiplication) is known to be finite (see for example \cite{And98}). Further, with the asymptotic behaviour of CM points (with respect to discriminant) having been studied by Herrero, Menares and Rivera-Letelier, the remaining special loci in $C$ to understand are the intersections of $C$ with modular curves, as in Theorem \ref{Main}.

	In the ordinary case, further analysis of the locus $L=C\cap \bigcup Y_0(n)$ will examine where the accumulation points lie. In particular, the set can be non-discrete. This locus will be studied by introducing some analytic curves which look similar to modular curves on deformation discs. These will be called \textit{tori}. A formal definition is given in Definition \ref{tori}, inspired by the appearance of modular curves in Serre-Tate coordinates. Those that actually are branches of modular curves will be called \textit{modular tori}.
	
	\begin{thm}\label{CasesThm}
		Let $P=(j_1,j_2) \in Y(1)\times Y(1)$ be a point on a curve $C$ defined over $\cp$ corresponding to elliptic curves $E_1,E_2/$ defined over $\cp$ with good ordinary reduction such that their reductions $\bar{E}_1$ and $\bar{E_2}$ are isogenous. Then 
	\begin{itemize}
		\item[(a)] if neither $E_1$ nor $E_2$ has CM, then $P$ is an accumulation point of $L$ if and only if $P$ lies on a torus.
		\item[(b)] if exactly one of $E_1$ and $E_2$ has CM, then $P$ is not an accumulation point of $L$.
		\item[(c)] if both $E_1$ and $E_2$ have CM then $P$ is an accumulation point of $L$ if and only if $C$ is tangent at $P$ to a torus passing through $P$.
	\end{itemize}
	\end{thm}
	Note that the hypotheses on $C$ in Theorem \ref{Main} (that $C$ is smooth and has irreducible reduction) are not required for this Theorem. It follows from this characterisation that for any curve $C$, the set of $C$'s intersection with the special locus is never discrete (see Corollary \ref{always}).
	
	Of course, if $P$ is an accumulation point then $\bar{E}_1$ and $\bar{E_2}$ are isogenous: in particular $P$ is within a distance $1/p$ of the locus in question, so there are $E_1'\equiv E_1 \pmod{p}$, $E_2'\equiv E_2\pmod{p}$ with $(E_1',E_2')\in C\cap \bigcup_nY_0(n)$. This means that $E_1'$ and $E_2'$ are isogenous, hence so are their reductions $\bar{E}_1$ and $\bar{E}_2$. Thus there is no restriction in the initial assumption that the reductions are isogenous.

	This description allows the results from explicit methods here to be related to a special case of previous work of Davesh Maulik and Bjorn Poonen. In \cite{MP09}, they investigate the jumping loci of various families of varieties. In comparison with Theorem \ref{CasesThm} here, their work (specifically, Proposition 1.13 there) says that the collection of points on $C$ where either one coordinate has CM or the two coordinates correspond to a pair of isogenous elliptic curves is $p$-adically nowhere dense in $C$. This is implied by Theorem \ref{CasesThm}, as discussed in Section \ref{Relate}. While it is not the case that the equidistribution result of Theorem \ref{Main} implies this nowhere density result (nor vice-versa), and the methods used to prove Theorem \ref{Main} also turn out to be largely $p$-adic in nature, the statements about equidistribution capture more about the asymptotic distribution of the special points on $C$ by expanding the viewpoint to the larger Berkovich space. Note, however, that the results from \cite{MP09} do not require any assumptions on ordinariness, whereas for the techniques used in the proof of \ref{CasesThm} here, ordinariness is essential.

	\subsection{Outline of the Proof}
	
	Following Herrero, Menares and Rivera-Letelier
	(Lemma 2.3 in \cite{HMR20}), the condition of weak convergence of measures in Theorem \ref{Main}, which is phrased in terms of Berkovich space, may be reinterpreted as follows, making reference only to $\cp-$points. Some background on the construction of the Berkovich space associated to a smooth curve is contained in \S\ref{Berk}. 	Let $C\subset \pee^1_{\cp}\times \pee^1_{\cp}$ be a smooth curve with integral model $\mathcal{C}$ and irreducible reduction $\bar{C}\colonequals\mathcal{C}_{\fpbar}$. Let $\Lambda_n$ be a sequence of finite non-empty sub-multisets of $C(\cp)$ (that is, elements may appear with multiplicity). Also view $\Lambda_n$ inside teh analytification $C^\text{Berk}$ under the inclusion $C(\cp)\hookrightarrow C^\text{Berk}$. For each $n$, let $\bar{\delta}_n$ be the normalised probability measure supported on $\Lambda_n$, accounting for multiplicity, and denote by $\delta_{x_\text{can}}$ the measure supported at the unique canonical point of $C^\text{Berk}$. In the Lemma below, the following notation will be used:
	
	\begin{defn}
		Let $r>0$ and $a_1,a_2 \in \cp$. Then the set $B^\infty_r((a_1,a_2))$ is defined as \[B^\infty_r((a_1,a_2))=\{(x_1,x_2)\in \cp^2\ | \ |x_1-a_1|_p>r \text{ or } |x_2-a_2|_p>r\}.\]
	\end{defn} 
	
	\begin{lemma}\label{padiccondition}
	 The sequence of measures $\bar{\delta}_n$ converges weakly to $\delta_{x_\text{can}}$ if and only if for any open ball of the form $B=B_r((a_1,a_2))$ of radius $r<1$, or $B = B^\infty_R((a_1,a_2))$ for $R>1$,
		\begin{align}\label{Limit}
			\lim_{n\to\infty}\frac{|\Lambda_{n}\cap B|}{|\Lambda_{n}|}= 0.
		\end{align} Here, $|S|$ means the cardinality of the multi-set $S$ (i.e. counting with multiplicity).
	\end{lemma}

	\begin{rem}
		The condition that $\bar{C}$ is irreducible ensures that there is a unique canonical point, else there would be one for each component as well as other measures arising from the measures supported on the skeleton of the analytification.
	\end{rem}

	This turns the question of equidistribution in Theorem \ref{Main} into a question about bounding intersections of $C$ with modular curves on residue discs (or, more precisely, on any disc which is strictly smaller than a residue disc). The aim will then be to apply Lemma \ref{padiccondition} to the multisets $(C\cap Y_0(n))_n$, which by B\'ezout's Theorem have respective sizes deg$(C)\cdot d_n$, for some numbers $d_n\geq n$, the degrees of the curves $Y_0(n)$. So it will suffice to show that, for any fixed ball $B$ of radius strictly less than 1, or for balls at infinity as described in the statement of Lemma \ref{padiccondition},
	\begin{align}\label{ordseq} 
		\lim_{n\to\infty}\frac{|C\cap Y_0(n)\cap B|}{n}= 0\end{align}	
	in the ordinary case, or that 
	\begin{align}\label{ssingsubseq}
		\lim_{n\to\infty}\frac{|C\cap Y_0(a_n)\cap B|}{a_n}= 0\end{align} in the supersingular case, for any sequence $a_n$ with $v_p(a_n)\to \infty$.

	\subsection{Techniques in the Ordinary Case}
	
	 Around points on $C$ where both coordinates correspond to elliptic curves with ordinary reduction, the analytic change of coordinates to Serre-Tate coordinates (described in \S\ref{ST}) is defined and makes the curves $Y_0(n)$ far easier to deal with. In $j$-coordinates, the equations for modular curves become unwieldy very quickly. For instance, the equation cutting out $Y_0(2)$ already has 7 nonzero coefficients and the constant term has 15 decimal digits. See \S3 of \cite{BOS16}. In Serre-Tate coordinates, however, the $Y_0(n)$ are unions of translates of formal tori by $p$-power roots of unity on residue discs. These explicit equations, and the fact that the Serre-Tate change of coordinates has integral coefficients, turn the question into an explicit analysis of the zeros of some $p$-adic power series. Strassman's Theorem in non-Archimedean analysis (see \ref{Strass} for the statement) can then be used to bound the number of zeros.
	
	In the study of accumulation points in Section \ref{Cases}, the question still requires only local information, so Serre-Tate coordinates may be used here too. This time the focus is on finding intersection points, rather than bounding the total number of intersection points, and so the lower bounds from Strassman's Theorem, which apply over algebraically closed fields such as \cp, are used. A key step is Lemma \ref{approxlemma}, which checks that the collection of exponents appearing in modular tori is dense in the collection of all possible exponents. This is used to exhibit points on tori as accumulation points. The section concludes with the observation that there is an accumulation point of the special locus on any curve $C$.

	\subsection{Techniques in the Supersingular Case}
	
	In the supersingular case, properties of the period map will be used instead. There is a period map defined on any ball of radius strictly less than 1 inside a supersingular residue disc. It can be thought of as a map of rigid analytic spaces, and it collapses $p$-power Hecke orbits. Thus, fixing $R<1$, it is possible to bound the number of $m$ for which $Y_0(p^m)$ intersects a ball of radius $R$. Then (as long as $R$ is large enough), since prime-to-$p$ Hecke orbits preserve distance, there is uniform control over the $p$-divisibility of $n$ such that $Y_0(n)$ intersects the ball of radius $R$ at all. In particular, the equidistribution result follows from the stronger statement that, for divisible enough $n$ (what is `enough' depends on $R$), $Y_0(n)$ misses this ball completely and so the numerator in (\ref{ssingsubseq}) is eventually 0. The section concludes with a family of examples not satisfying the hypotheses $v_p(a_n)\to\infty$ where equidistribution fails.

	\subsection{Structure}
	
	In the ordinary setting, most of the calculations, including those which give a concrete description of the modular subvarieties $Y_0(n)$, will be undertaken in Serre-Tate coordinates, which are reviewed in Section \ref{ST}.

	In Section \ref{ModularCurves}, this description of modular curves in the ordinary locus will be obtained, using Serre-Tate coordinates. This was previously worked out by Rutger Noot  (\cite{Noo96}) and Ben Moonen (\cite{Moo98II}). In Section \ref{FinishProof}, this description is used to complete the evaluation of the limit in (\ref{ordseq}) for balls not containing infinity.
	
	Section \ref{Ex1} begins a study of non-discreteness by constructing an example in Serre-Tate coordinates of a curve where the locus $C\cap \bigcup_n Y_0(n)$ has accumulation points. This is then expanded in Section \ref{Cases} to analyse exactly when a point on an algebraic curve $C$ is an accumulation point of the locus $C\cap \bigcup_n Y_0(n)$, resulting in Theorem \ref{CasesThm}. In Section \ref{Relate} this result is related to the aforementioned work of Maulik and Poonen, \cite{MP09}

	For the supersingular setting, the construction and basic properties of the period map are reviewed in Section \ref{Period}, and the result for the subsequences specified in (\ref{ssingsubseq}) is deduced, using some results about Hecke orbits which are referenced in Section \ref{Hecke}.

	The argument will then be completed by verifying the conditions at infinity in \S\ref{infinity} for balls of the shape described in the statement of Lemma \ref{padiccondition}. This will be done using Tate coordinates coming from $p$-adic uniformisation.

	\subsection{Notation and Conventions}\label{Notation}
	
	The absolute value $|\cdot|$ on $\cp$ is normalised so that $|p|=\frac{1}{p}$, and the valuation $v_p$ satisfies $v_p(p)=1$. The maximal ideal of $\cp$ will be denoted $\mathfrak{m}_p$ and its valuation ring $\oh_p$.
	
	$K$ will denote a finite extension of the field $K_0=W(\fpbar)[1/p]$, with ring of integers \ok\ and maximal ideal $\mathfrak{m}_K$.
	
	If $L$ is a field and $I\triangleleft L[x_1,...,x_n]$ is an ideal, denote by $\mathbb{V}(I) \subset \ay^n_L$ the common vanishing of the polynomials in $I$. $\mathbb{V}(f)$ will be understood to mean $\mathbb{V}((f))$.
	
	The measure $\bar{\delta}_n(C)$ is defined by
	
	\[\bar{\delta}_n(C)= \frac{1}{\text{deg}(C)\text{ deg}(Y_0(n))} \sum_{P\in Y_0(n)\cap C} \text{mult}(P)\delta_P,\] 
	
	where mult$(P)$ is the degree of the intersection of $C$ and $Y_0(n)$ at $P$.
	
	The $j$-invariant of an elliptic curve will sometimes be conflated with a representative for the corresponding isomorphism class. Further, when looking at open balls of radius 1 in $j$-coordinates, since every point in the ball is a centre, it is not ambiguous to denote the centre by the common reduction modulo $\mathfrak{m}_p$ of all elliptic curves in that ball (even though this is not actually a point in that ball). Similarly, in the ordinary case, since Lemma \ref{padiccondition} requires finding a bound on all ordinary balls $B$ (i.e. with ordinary centre) of radius $R<1$, after perhaps increasing the radius of a ball to some $R'<1$, the ball may be assumed to contain the canonical lift of the common reduction of the elliptic curves in $B$. $B_R(a)$ will denote the open ball of radius $R$ centred at $a$, and $\bar{B}_R(a)$ will denote the corresponding closed ball. Two-dimensional balls will use the metric $d((a_1,b_1),(a_2,b_2)) = \max\{|a_1-a_2|,|b_1-b_2|\}$.
	
	For an isogeny $\bar{f}\colon \bar{E}_1\to\bar{E}_2$, the number $c_{\bar{f}}$ will be defined by the ratio \[c_{\bar{f}}=\frac{a_{\bar{f}^T}}{a_{\bar{f}}},\] where $a_{\bar{f}}\in \zed_p$ is the endomorphism of $\zed_p$ corresponding to the map induced by $\bar{f}$ on the \'etale part of p-divisible groups. A more careful definition of $a_{\bar{f}}$ is given in Section \ref{ST}, and the significance of the numbers $c_{\bar{f}}$ will be explained in Section \ref{ModularCurves}.
	
	Throughout, except in \S\ref{Cases}-\ref{Relate} where the assumptions are not necessary because the canonical point is not being discussed, the curve $C$ will be smooth, have irreducible reduction and (unless otherwise stated) not be a Shimura subvariety of the product of modular curves.

	\subsection{Acknowledgements}
	
	The author is very grateful to Ananth Shankar for proposing the problem and providing much support throughout. The author would like to thank Apoorva Agarwal, Haochen Cheng, Klaus K\"unnemann, Jackson Morrow, Riccardo Pengo, Micha\l \ Szachniewicz, Salim Tayou and Jit-Wu Yap for helpful conversations.
	
	This work was partially supported by the NSF grant DMS-2338942.

	\section{Background on Berkovich Analytification}\label{Berk}
	
	The background material in this section can be found in Matthew Baker's notes \cite{Bak07} and Vladimir Berkovich's book \cite{Ber90}.
	
	Let $C$ be a smooth proper integral curve over $\cp$. The Semistable Reduction Theorem says that $C$ admits a formal model $\mathcal{C}/\oh_p$ where no singularities appear except for ordinary double points. Write $\bar{C}$ for the special fibre of $\mathcal{C}$ and assume $\bar{C}$ is irreducible.
	
	The reduction map $\pi\colon C(\cp)\to \bar{C}(\fpbar)$ extends to a map
	\[r\colon C^\text{Berk}\to \bar{C},\] with target the scheme-theoretic points of $\bar{C}$. Under the assumption that $\bar{C}$ is irreducible, there is a unique point, $x_\text{can}$ in $C^\text{Berk}$ reducing to the generic point of $\bar{C}$ under the map $r$. A basis of neighbourhoods for this point consists of the complement of finitely many discs of radius strictly less than 1.
	
	\begin{lemma} \label{padicconditionproof}
		Let $\Lambda_n$ be divisors on $C$. Suppose for every ball $B$ of radius $R<1$ in $\pee^1\times \pee^1$ (using the supremum metric), the limit 
		\[\lim_{n\to \infty} \frac{\text{deg}(\Lambda_n|_B)}{\text{deg}(\Lambda_n)}\]
		is zero. Then $\bar{\delta}_{\Lambda_n}\to \delta_{x_\text{can}}$. Conversely, equidistribution holds only if all such limits are zero.
	\end{lemma}
	
	\begin{proof}
		The ideas are similar to those in \cite{HMR20}, Lemma 2.3. To test convergence of measures, it suffices to check that for every continuous $g\colon C^\text{Berk}\to \arr$ (automatically bounded as the analytification of a proper scheme is compact) that 
		\[\int g\ d\bar{\delta}_{\Lambda_n}\to \int g\ d\delta_{x_\text{can}}\]
		
		First, control $g$ near $x_\text{can}$: fixing $\epsilon>0$, by continuity of $g$ there exists a neighbourhood $U$ of $x_\text{can}$ such that for $x\in U$, $|g(x)-g(x_\text{can})|<\epsilon$. $U$ may be taken to have form the complement in $C^\text{Berk}$ of finitely many (one-dimensional) discs $\{B_{r_i}(p_i)\}_{i=1}^N$ for some points $p$ in $C$, since such sets form a base. Writing $g(D)$ for the sum of the evaluations of $g$ on points in $D$ (counting with multiplicity), have the estimates	
		
		\begin{align*}
			\left|\int g\ d\bar{\delta}_{\Lambda_n}- \int g\ d\delta_{x_\text{can}}\right| &= \left|\frac{g(\Lambda_n)}{\text{deg}(\Lambda_n)} - g(x_\text{can})\right| \\
			&\leq \left|\frac{g(\Lambda_n|_U)}{\text{deg}(\Lambda_n)} - \frac{g(x_\text{can})\text{deg}(\Lambda_n|_U)}{\text{deg}(\Lambda_n)}\right| + \left|\frac{g(\Lambda_n|_{U^c})}{\text{deg}(\Lambda_n)} - \frac{g(x_\text{can})\text{deg}(\Lambda_n|_{U^c}))}{\text{deg}(\Lambda_n)}\right|		\\
			&\leq \epsilon + 2\sup{g} \sum_{i=1}^N\frac{\text{deg}(\Lambda_n|_{B_{r_i}(p_i)})}{\text{deg}(\Lambda_n)}
		\end{align*}
		where the latter terms go to zero in $n$ by hypothesis (the hypothesis checks this ratio on balls in the larger space $\pee^1\times\pee^1$).
		
		For the converse, suppose $\bar{\delta}_{\Lambda_n}\to\delta_{x_\text{can}}$. Let $B$ be a ball of radius $R<1$ and centre $P$ on $C$. Choose an open affinoid patch $U$ containing both $B$ and $x_\text{can}$. This can be arranged by the above description of a base of neighbourhoods of the canonical point. In a coordinate $T$ on this patch, define $\alpha: U \to \arr$ by $\alpha(x) = x(T-a_P)$, where $a_P\in \cp$ corresponds to $P$ with respect to this coordinate. Then as in \cite{HMR20}, Lemma 2.3, choose $\phi:[0,\infty)\to \arr$ continuous with $\phi([0,R])=1$ and $\phi(1)=0$ and set $F=\phi\circ\alpha$. Then \[0\leq \bar{\delta}_{\Lambda_n}(B)\leq \int F d\bar{\delta}_{\Lambda_n},\] but since $\bar{\delta}_{\Lambda_n}\to\delta_{x_\text{can}}$ and $F(x_\text{can})=\phi(1)=0$, the limit in $n$ of the right hand side is 0 as desired.
	\end{proof}

	In what follows, the results of this section can be applied to $C\subset X(1)\times X(1)$ smooth and with irreducible reduction. To extend this beyond the cases where $\bar{C}$ is irreducible, the limiting measure, if still supported on the skeleton of $C^\text{Berk}$, could be a weighted sum of Dirac measures on the canonical points (of which there is one for each irreducible component of $\bar{C}$), or involve non-discrete measures supported on the edges. If indeed there is equidistribution in these cases then - once the correct limiting measure is identified and an extension of Lemma \ref{padicconditionproof} is proven - the techniques in the following sections will all apply equally well since the proof of Lemma \ref{padicconditionproof} is the only place where irreducibility of $\bar{C}$ is being used.

	\section{The Ordinary Locus} \label{SecOrd}
	
	\subsection{Serre-Tate Coordinates}\label{ST}
	
	The main results stated in this subsection can be found in \textit{Serre-Tate Local Moduli} by Nick Katz.
	
	Fix a finite extension $K/K_0$. Let $\bar{E}$ be an elliptic curve over $\fpbar$, with canonical lift ${\bar{E}}^{\text{can}}$. The collection of lifts of $\bar{E}$ are described by the following theorem:
	
	\begin{thm}
		There is an analytic bijection $\Phi_{\bar{E}}$ with inverse
		\[\{j\in \ok \ | \ j\equiv j(\bar{E})\pmod{\mathfrak{m}_K} \}\xrightarrow{\Phi_{\bar{E}}^{-1}} 1+\mathfrak{m}_K\]
		sending $j({\bar{E}}^{\can})$ to $1\in 1+\mathfrak{m}_K$.
	\end{thm}
	
	In particular, the collection of lifts of $\bar{E}$ has a group structure, given by multiplication of units in $\ok$ which are congruent to 1 modulo $\mathfrak{m}_K$.
	
	Having described lifts of a single elliptic curve ${\bar{E}}$, the next result gives information about lifting isogenies between elliptic curves. An isogeny $\bar{f}\colon\bar{E}_1\to \bar{E}_2$ between elliptic curves over \fpbar\ induces a map $\bar{f}[p^\infty]\colon\bar{E}_1[p^\infty](\fpbar) \to\bar{E}_2 [p^\infty](\fpbar)$ on the \'etale parts of their $p$-divisible groups. This is a map between two rank 1 $\zed_p$-modules, so it is given by multiplication by some $p$-adic integer. Denote this number by $a_{\bar{f}}\in \zed_p$. Functoriality of this construction implies that $a_{\bar{f}} a_{\bar{f}^T} =$ deg$(f)$, where $\bar{f}^T$ is the dual of $\bar{f}$. 
	
	\begin{thm}\label{toruscondition}
		Let $\bar{f}\colon\bar{E}_1\to \bar{E}_2$ be an isogeny between elliptic curves over \fpbar. Let $E_1,E_2$ be lifts of $\bar{E}_1, \bar{E}_2$ respectively, and let $q_i = \Phi_{\bar{E}_i}^{-1}(j(E_i))$ be their corresponding Serre-Tate coordinates. Then $\bar{f}$ can be lifted to $f\colon E_1\to E_2$ if and only if 
		\[q_1^{a_{\bar{f}^T}} = q_2^{a_{\bar{f}}}\]
	\end{thm}
	
	Since the possible endomorphism rings of ordinary elliptic curves can be described (they are orders in imaginary quadratic extensions of \kew\ with conductor prime to $p$), the numbers $a_{\bar{f}}$ which appear can be understood explicitly, see Section \ref{ModularCurves}. From this description, a lift $E$ of $\bar{E}$ has complex multiplication if and only if there exists $\bar{f}\in$ End($\bar{E}$) lifting to an endomorphism of $E$ with the property that $a_{\bar{f}}\ne a_{\bar{f}^T}$. The following well-known result is related to this observation and Theorem \ref{toruscondition}:
	
	\begin{thm}
		The elements of finite order in $1+\mathfrak{m}_K$ correspond under the Serre-Tate bijection to $j$-invariants of elliptic curves over $\ok$ lifting $\bar{E}$ which have complex multiplication.
	\end{thm} 
	
	 That is, the natural finite collection of elements of finite order in the group $1+\mathfrak{m}_K$ (the $p$-power roots of unity in the discretely valued field $K$) correspond to CM lifts of $\bar{E}$ under the Serre-Tate bijection $\Phi_{\bar{E}}$.

	The same equations will cut out the same loci after base changing to \cp, hence the calculations of the following sections can be carried out for $C$ defined over $\cp$ and not just a finite extension of $K_0$.

	\subsection{Description of the Image of $Y_0(n)$ In Serre-Tate Coordinates}\label{ModularCurves}
	
	This section will begin a study of $C$ on residue discs around ordinary points. The results here are related to the work of Rutger Noot in his thesis (\cite{Noo96}) and the paper \textit{Models of Shimura Varieties in Mixed Characteristic} (see Theorem 2.8 there), and to the work of Ben Moonen in \cite{Moo98II} (see Theorem 4.5 there).
	
	Suppose now that $C$ is given by an equation $F(j_1,j_2)=0$ in the product of two $j$-lines, with $F\in \oh_p[X,Y]$. Continue to denote by $\Phi_{\bar{E}}$ the analytic function taking a Serre-Tate coordinate $q\in 1+\mathfrak{m}_p$ to the $j$-invariant of the associated lift $E/ \oh_p$ of an ordinary elliptic curve $\bar{E}/\fpbar$, as set up in Section \ref{ST}. $\Phi_{\bar{E}}$ is an isomorphism from $1+\mathfrak{m}_p$ to the $\cp$-points of the open ball $B_1({\bar{E}}^{\text{can}})$. 
	
	Given $\bar{E}_1$, $\bar{E}_2$ and $0<R<1$, let $q_1,q_2$ be their respective Serre-Tate coordinates, each running over $B_R(1)$. Set $s_i = q_i-1$, with ranges $|s_i|<R$.

	Given $\bar{E}_1$ and $\bar{E}_2$, define
		\[\Lambda_{n,R} = \{(q_1,q_2)\in B_R((1,1)) \ | \ (\Phi_{\bar{E}_1}(q_1),\Phi_{\bar{E}_2}(q_2))\in Y_0(n)\},\]
	the locus of points $(q_1,q_2)\in B_R((1,1))$ corresponding to lifts of $E_1$ and $E_2$ which lie on $Y_0(n)$. Then the numerator in the limit in Equation (\ref{Limit}) can be bounded by counting the number of solutions to the equation $F(j_1,j_2)=0$ in the locus $\Lambda_{n,R}$. This section is devoted to a description of $\Lambda_{n,R}$. Intersections with $C$ will be considered in the next section.

	\begin{lemma}\label{listoftori}
		For fixed ordinary $\bar{E}_1$,$\bar{E}_2$, there is a constant $C_1$, independent of $n$, such that $\Lambda_{n,R}$ is a union of translates of at most $C_1\sqrt{n}$ tori by $p$-power roots of unity which lie in $B_R(1)$.
	\end{lemma}
	
	\begin{proof}
		If $\bar{E_1}$ and $\bar{E_2}$ are not isogenous then $\Lambda_{n,R}$ will be empty. So assume they are isogenous and fix $g\colon\bar{E_1}\to \bar{E_2}$ an isogeny of minimal degree $m$, say. Then $\bar{E_1}\cong \bar{E}_2/$ker$g^T$, where $g^T\colon\bar{E}_2\to\bar{E}_1$ denotes the dual of $g$. There is a map Hom$(\bar{E}_1,\bar{E}_2)\to$ End$(\bar{E}_2)$ given by $h\mapsto h\circ g^T$.
		
		Since non-zero isogenies are surjective, this map is an injection. The image of this map is \[\{f\in \text{End}(\bar{E_2}) \ | \text{ker}(f)\supset \text{ker}(g^T)\},\] a left ideal in End$(\bar{E}_2)$. Since $\bar{E}_2$ is ordinary, its endomorphism ring is an order in the ring of integers of an imaginary quadratic field, say $\kew(\sqrt{-d})$ for $d>0$ depending on $\bar{E}_2$. By ordinariness again, $\kew(\sqrt{-d})$ can be viewed as a subfield of $\kew_p$.

		Recall that the degree, or size of kernel, of the isogeny corresponding to $a+b\sqrt{-d}\in$ End$(\bar{E}_2)$ is given by its norm as an algebraic number, which is $a^2+b^2d$. So if $h\circ g^T = a+b\sqrt{-d}$ as elements of End$(E_2)$,
		\[\text{deg}(h\colon\bar{E}_1\to\bar{E}_2) = \frac{N(a+b\sqrt{-d})}{|\text{ker}(g^T)|} = \frac{a^2+b^2d}{m}.\]
		Thus the number of $n$-isogenies between $\bar{E_1}$ and $\bar{E_2}$ is bounded by the number of solutions in half integers $a,b$ to $a^2+b^2d=mn$. Since $|b|\leq \sqrt{mn}$ and determines $a$ up to sign, there can be at most $4\sqrt{mn}+2\leq C_1\sqrt{n}$ solutions, for a positive constant $C_1$ ($m$ is fixed once the ball is fixed). So $\Lambda_{n,R}$ can be written as the union over at most $C_1\sqrt{n}$ $n$-isogenies $\bar{f}$ of the sets \[\Lambda_{n,R,\bar{f}} = \{(E_1,E_2) \ | \text{ there exists } f\colon E_1\to E_2 \text{ reducing to } \bar{f}\colon\bar{E}_1\to\bar{E}_2\}.\]

		Fix $n$ and let $q_1=\Phi_{\bar{E}_1}^{-1}(j(E_1)),\ q_2=\Phi_{\bar{E}_2}^{-1}(j(E_2))\in 1+\mathfrak{m}_p$ be the Serre-Tate coordinates of lifts $E_1$, $E_2$ of $\bar{E_1}$, $\bar{E_2}$ respectively. By Theorem \ref{toruscondition}, pairs $(E_1,E_2)$ lifting a fixed such $n$-isogeny $\bar{f}\colon\bar{E}_1\to\bar{E}_2$ are exactly those satisfying $q_1^{n/a_{\bar{f}}} = q_2^{a_{\bar{f}}}$. 
		
		If $(n,p)=1$, $a_{\bar{f}}\in \zed_p^\times$ and so this condition is equivalent to $q_2 = q_1^{n/a_{\bar{f}}^2}$. Thus $\Lambda_{n,R,\bar{f}}$ is exactly a torus.
		
		If $(n,p)\ne 1$, noting first that $\bar{f}$ lifts to $E_1\to E_2$ if and only if $\bar{f}^T$ lifts to $E_2\to E_1$, assume that $f$ is such that $v_p(a_{\bar{f}})\leq v_p(\frac{n}{a_{\bar{f}}})$. Writing $t=v_p(a_{\bar{f}})$, have \[q_1^{n/a_{\bar{f}}} = q_2^{a_{\bar{f}}}\Leftrightarrow \text{ there exists a }(p^t)^\text{th} \text{ root of unity } \zeta \text{ such that } q_2= \zeta q_1^{n/(a_{\bar{f}}^2)},\] with $n/(a_{\bar{f}}^2)\in\zed_p$. That is, the locus $\Lambda_{n,R,\bar{f}}$ is a collection of cosets of a fixed torus indexed by $p$-power roots of unity. But recall that $B_R(1)$ is a closed subgroup of $1+\mathfrak{m}_p$, so if $q_1, q_2$ lie in $B_R(1)$, then so do any limits of integral powers of $q_1$ and $q_2$. Hence \[\zeta = q_2q_1^{-n/(a_{\bar{f}}^2)}\] lies in $B_R(1)$ as desired.
	\end{proof}

	In particular, the number of $p$-power roots of unity in any ball $B_R(1)$ is finite. Denoting this number by $u_R$, there are at most $C_1u_R\sqrt{n}$ torus translates comprising $\Lambda_{n,R}$ for some constant $C_1$ independent of $n$.

	\begin{rem} \label{CModular}
		The assumption throughout (including in what follows) is that $C$ is not a Shimura subvariety of $X(1)\times X(1)$. However if $C=Y_0(m)$, say, then within the ordinary locus, intersection points of $\Lambda_{n,R}$ with $C$, for $n>m$ squarefree, are given by simultaneously solving equations of the form $q_2 = q_1^r$ and $q_2=q_1^s$ with $r\ne s\in \zed_p$. Hence solutions are roots of $q_1^{r-s}=1$, i.e. parametrised by roots of unity. So the number of intersections of any given pair is at most $u_R$, whence
			\[\lim_{n\to\infty \text{ sqfree}}\left|\frac{\text{deg}(\Lambda_{n}|_B)}{\text{deg}(\Lambda_{n})}\right|\leq \lim_{n\to\infty}\frac{C_1\sqrt{n} u_R}{n}=0\]
			
		The same does not work for the supersingular locus. For instance, if $C=Y_0(1)$ is the diagonal, fix a supersingular elliptic curve $\bar{E}$ whose endomorphism ring is a maximal order. Then computing $|Y_0(1)\cap Y_0(l)|$ for prime numbers $l\ne p$ in the residue disc over $\bar{E}$ is to ask how many lifts of $\bar{E}$ have self $l$-isogenies. 
		Deuring's Lifting Theorem says that, given a rank 2 integrally closed subalgebra $\oh$ of End$(\bar{E})$, there exists a lift of $\bar{E}$ to $W(\fpbar)$ with endomorphisms by $\oh$ (see Proposition 2.3 in \cite{OS19}). Fixing \oh, there can be at most six elements of norm $l$ in $\oh$ (corresponding to the number of units in the quadratic imaginary field containing \oh). Hence there are as many intersection points as one-sixth of the number of elements in the set
		\[\{(a,b,c,d)\in \zed^4 \ | \ a^2+b^2+pc^2+pd^2 = l\},\] (that is, the set of elements of norm $l$ in the quaternion algebra End$(\bar{E})$). The size of this set is linear in $l$, and deg$Y_0(l)=l+1$
		 Hence 
		\[\limsup_{\text{primes }l}\frac{|Y_0(1)\cap(Y_0(l))\cap B_{1/p^{1/2}}(\tilde{E},\tilde{E})|}{\text{deg}(Y_0(1))\text{deg}(Y_0(l))}>0\] and equidistribution fails by the converse to Lemma \ref{padicconditionproof} for the curve $Y_0(1)$.
	\end{rem}

	\subsection{Counting Intersections with a Fixed Torus} \label{FinishProof}
	
	Again, in this section, only the behaviour of $C$ within in the ordinary locus will be studied.
	
		After changing coordinates to Serre-Tate coordinates, the equation for $C$ will be given by the vanishing of a two-variable power series (rather than a polynomial). One of the coordinates of a point of intersection with a torus should be given by a root of a 1-variable power series, with the other coordinate calculated by referring back to the equation cutting out the torus. The main tool for counting roots of power series will be the following Theorem, whose proof can be found in Section 6 of \cite{Gou93}:
	
	\begin{thm} (Strassman's Theorem) \label{Strass}
		Let $f(X) = \sum_{n\geq 0} a_nX^n$ be a non-zero power series with coefficients $a_n\in\oh_p$ satisfying $|a_n|\to 0$ as $n\to \infty$. Let $N$ be maximal with the property that $|a_N|\geq |a_n|$ for all $n\geq 0$ ($N$ exists since $|a_n|\to 0$ but not all are zero). Then there exist exactly $N$ roots $x\in \oh_p$ of $f(x)=0$ (counted with multiplicity).
	\end{thm}
	
	To apply the Theorem, it is necessary to carefully convert the problem into finding zeros of a single variable power series which converges on $B_1(0)$. A warning sign (see \cite{Gou93}, discussion after Theorem 4.3.3) is that when composing the power series $f(X) =\exp(X)$ and $g(X)=(2X^2-2X)$, the composite function $f\circ g$ gives a different answer with input $1$ from the evaluation of the formal composition of power series at $X=1$.

	For $F(X,Y)=\sum_{i\geq1} a_{ij}X^iY^j\in K[[X,Y]]$, and $\epsilon>0$, say that $f$ is \textit{$\epsilon$-convergent} if \[\lim_{i+j\to\infty}\epsilon^{i+j}|a_{ij}|\to 0\] (i.e. any non-repeating sequence of the real numbers $\epsilon^{i+j}|a_{ij}|$ has limit zero). For such $F$, define \[||F||_{\epsilon} = \max_{i,j} \epsilon^{i+j}|a_{ij}|.\] The same definitions also make sense for power series in one variable (they are just two variable power series which have $a_{ij}=0$ if $j\ne0$). It is a fact (Proposition 5.3 in \cite{Sch11}) that if $f,g$ are $\epsilon$-convergent, then $||fg||_\epsilon=||f||_\epsilon||g||_\epsilon$. The following Proposition is useful for understanding when na\"ive compositions of power series converge; it is Proposition 5.4 in \cite{Sch11}.

	\begin{prop} \label{epsdel}
		Let $F \in K[[X,Y]]$, $H_1, H_2\in K[[T]]$ and suppose $F$ is $\epsilon$-convergent and $H_1,H_2$ are $\delta$-convergent for some $\delta$ with $\max\{||H_1||_\delta,||H_2||_\delta\}\leq\epsilon$. Then for any $s,t\in K$ with $|s|,|t|\leq \delta$, the composite power series $F(H_1(S),H_2(T))$ converges at $(S,T)=(s,t)$ to the value of the composition of functions $F(H_1(s),H_2(t))$.
	\end{prop}

	Let $F(J_1,J_2)\in \oh_p[J_1,J_2]$ be the polynomial cutting out $C$. Suppose that intersections of $C$ with the locus $q_1^a = q_2^b$ is being considered. By symmetry, may assume that $v_p(a)-v_p(b)=e\geq 0$. That is, the locus is given by equations of the form $q_2 = \zeta q_1^{a/b}$ for some $(p^e)^\text{th}$ root of unity $\zeta$. Recentre by setting $q_2=\zeta+s_2$ and $q_1 = 1+s_1$, so that the coordinates $s_1,s_2$ both run over $B_1(0)$. 
	
	Fix a point $P=(j_1,j_2)$ on $C$ where both coordinates have ordinary reduction. Let $J_1 = H_1(S_1)$, $J_2 = H_2(S_2)$ be the analytic change of coordinates into Serre-Tate coordinates (shifted as above to be centred at zero) for the residue disc of $P$. Substituting in, recover a power series $F_\zeta(S_1,S_2)=0$ cutting out $C$ in the residue disc of $P$. Since $F$ itself is a polynomial, it is $\epsilon$-convergent for every $\epsilon>0$. As $H_1$ and $H_2$ have integral coefficients, they are both $\delta$-convergent for any $\delta<1$ and $||H_1||_1$ and $||H_2||_1$ are at most 1. Thus by Proposition \ref{epsdel}, the composition of series converges as would be expected for inputs $|s_1|,|s_2|< 1$. In particular, $F_\zeta$ is $\epsilon$-convergent for any $\epsilon<1$.
	
	In these coordinates, the torus $q_2 = \zeta q_1^{a/b}$ becomes
	\[s_2 = \zeta(1+s_1)^{a/b} -\zeta = \zeta \frac{a}{b} s_1 + \zeta \binom{a/b}{2}s_1^2 + ...,\]
	a power series with zero constant term. So, changing coordinates by translation to $S_1= Q_1-1$, $S_2 = Q_2-\zeta$, it is reasonable to substitute in $S_2 =  \zeta \frac{a}{b} S_1 + \zeta \binom{a/b}{2}S_1^2 + ... $ into the equation for $C$. This is to compute the composition $F_\zeta(G_1(S_1),G_2(S_1))$ where $G_1(S_1) = S_1$ and $G_2(S_1)=\zeta \frac{a}{b} S_1 + \zeta \binom{a/b}{2}S_1^2 + ...$. Given $R<1$, to solve for roots within the ball of radius $R$ is to be able to evaluate this power series on values $|s|<R$. Define $\epsilon=\max\{||G_1||_R,||G_2||_R\}<1$, then $F_\zeta$ is $\epsilon$-convergent. So by Proposition \ref{epsdel}, all convergence is as expected for $|s|<R$. 
	
	The upshot of this discussion is that power series and changes of coordinates can be composed freely in a na\"ive way below and the roots will correspond to meaningful solutions to questions about intersections of $C$ with torus translates.

	The aim is now to provide some uniform bounds on the possible $N$ (as in the statement of Strassman's Theorem) appearing from a family of power series $f$ of the specific form constructed above. The roots of these single variable equations will be studied in the following Lemma. Fix $R<1$ and $\zeta\in B_R(1)$ a $(p^e)^\text{th}$-root of unity. Let $F_\zeta(S_1,S_2)\in \oh_p[[S_1,S_2]]$ be the power series from above centred at $(1,\zeta)$ in Serre-Tate coordinates.

	\begin{lemma} \label{singletorus}
		 There exists $N=N(F_\zeta,R)$ such that for any $a,b\in \zed_p$ with $v_p(a)-v_p(b)=e$, the number of points in the vanishing locus $\mathbb{V}(F_\zeta)\subset \cp^2$ which lie both on the torus translate $(1+s_1)^{a}=(1+s_2)^b$ and in the ball $B_R(0)$ is at most $N$.
	\end{lemma}
	
	\begin{proof}
		The torus equation can be rewritten $s_2 = \zeta(1+s_1)^r-1$ for $r = \frac{a}{b}\in \zed_p$ and some root of unity $\zeta$ with $|1-\zeta|<R$.
		
		Write $F_\zeta(S_1,S_2) = a_{00}+\sum_{i+j\geq 1}a_{ij}S_1^iS_2^j$, and substitute in $S_2 = \zeta rS_1 + \zeta\binom{r}{2} S_1^2 + ...$ to get a power series in $S_1$ whose coefficients are all polynomials in $r$ with integral coefficients:
		\[F_\zeta(S_1,\zeta(1+S_1)^r-1) = c_0 + c_1(r)S_1 + c_2(r) S_1^2 + ... .\]		
		$F_\zeta$ is not the zero series, so there exists $m$ such that $c_m(r)$ is not the zero polynomial. Let $r_{1},...,r_M$ be the roots of this polynomial in $r$. That is, $c_m(r_i)=0$ for $i=1,...,M$ and $c_m(r)\ne0$ for all other $r$.
		
		For each $i=1,...,M$, since the vanishing locus of $F_\zeta$ does not contain any torus translates, there exists a coefficient $c_{m_i}(r_i)\ne0$. Since $|c_{m_i}(r)|$ is a continuous function of $r$, there exists $A_i\in \zed$ such that
		\begin{align}\label{cont}
			|r-r_i|<|\pi|^{A_i}\Rightarrow |c_{m_i}(r)|>\frac{|c_{m_i}(r_i)|}{2}\end{align}	
		Choose $\pi$ in $\cp$ with $R<|\pi|<1$ and consider finding roots of 
		\begin{align}\label{roots(r)}
			c_0 + c_1(r)\pi t + c_2(r)\pi^2t^2 + .... =0
		\end{align}		
		with $|t|\leq 1$ for various $r$. First suppose and $|r-r_i|<|\pi|^{A_i}$ for some $i$. Choose $N_i>m_i$ such that $|\pi|^{N_i}< \frac{|c_{m_i}(r_i)|}{2}|\pi|^{m_i}$, which is possible since $c_{m_i}(r_i)\ne0$ by the choice of $m_i$ ($N_i$ do depend on $\zeta$ which remains fixed as one of finitely many possibilities for now). It follows by Equation (\ref{cont}) that for all $n>N_i$, $|c_n(r)||\pi|^n<|c_{m_i}(r)||\pi|^{m_i}$. Therefore by Strassmann's Theorem there can be at most $N_i$ roots of (\ref{roots(r)}), independent of $r$ satisfying $|r-r_i|<|\pi|^{A_i}$.
		
		It remains to deal with $r$ satisfying $|r-r_i|\geq|\pi|^{A_i}$ for every $i$. This means that $r$ avoids a ball of fixed positive radius $|\pi|^A$ around every root, where $A=\max_i{A_i}$. It follows that $|c_{m}(r)|$ is bounded away from zero for such $r$: indeed if not, taking a sequence $(\rho_k)_{k\in \enn}$ of elements of $\zed_p$ with $|\rho_k-r_i|\geq |\pi|^{A}$ for every $k\in \enn$ and $i= 1,...,M$, such that $|c_{m}(\rho_k)|\to 0$, by compactness of $\zed_p$ there would be a limit point $\rho\in \zed_p$ with $c_{m}(\rho)=0$. This exactly means $\rho = r_i$ for some $i$, which is a contradiction since the sequence $\rho_k$ of which $\rho$ was a subsequential limit is bounded away from $r_i$ by at least a distance $|\pi|^{A}$. 
		
		Having established that the $|c_m(r)|$ are bounded below for such $r$, there exists $\tilde{N}>m$ such that \[0<|\pi|^{\tilde{N}}<\min_{\{r \ | \ \forall i,\ |r-r_i|\geq |\pi|^A \}} |c_{m}(r)| |\pi|^{m}.\] Then if $n>\tilde{N}$, $|c_n(r)||\pi|^n<|c_m(r)||\pi|^m$, so by Strassmann's Theorem, there can be at most $\tilde{N}$ roots of (\ref{roots(r)}), independent of $r$ satisfying that $|r-r_i|<|\pi|^{A_i}$ for all $i$.
		
		Since any $r$ either lies in a ball near one of the roots or avoids all such balls, \[N(F_\zeta,R) = \max\{N_1,...,N_M,\tilde{N}\}\] works independently of $r\in \zed_p$ coming from any equation $q_1^a=q_2^b$ with $v_p(a)-v_p(b)=e$.
	\end{proof}

	\begin{rem}
		$N(F_\zeta,R)$ does depend on $R$ through the absolute value of $\pi$.
	\end{rem}

	Recall that the root of unity $\zeta$ satisfied the inequality $|1-\zeta|<R$. In particular, there are only finitely many possible $\zeta$ (and also only finitely many possible $e$). Therefore the number\[N(F,R) = \max_{\{\zeta \ | \ |1-\zeta|<R\}}N(F_\zeta,R)\] is also finite.

	\begin{thm} \label{ThmOrd}
		The limit in (\ref{ordseq}) is 0 for any ball $B$ of radius strictly less than 1 with ordinary centres.
	\end{thm}
	
	\begin{proof}
		To compute the limit (\ref{ordseq}), the aim has been to count intersection points of modular curves with $C$ in a fixed ball $B_R((E_1,E_2))$, with $E_1,E_2$ lifts of ordinary elliptic curves and $R<1$. Suppose that $C = \mathbb{V}(F)$ for some polynomial $F\in \oh_p[x,y]$. By Lemma \ref{toruscondition}, this amounts to counting intersections with a collection of curves of the form $q_1^a=q_2^b$ in Serre-Tate coordinates. By Lemma \ref{listoftori}, at most $C_1u_R\sqrt{n}$ torus translates make up $Y_0(n)$ in Serre-Tate coordinates, where $C_1$ is a constant and $u_R$ is the number of roots of unity in $B_R(1)$. By Lemma \ref{singletorus} and the discussion immediately following it, no torus translate can have more than $N=N(F,R)$ intersection points with $C$. Thus the total size of the multiset $C\cap Y_0(n)\cap B_R((E_1,E_2))$ is at most $N(F,R)C_1u_R\sqrt{n}$, whence
		\[\left|\lim_{n\to\infty}\frac{\text{deg}(\Lambda_{n}|_B)}{\text{deg}(\Lambda_{n})}\right|\leq \lim_{n\to\infty}\frac{N(F,R)u_RC_1\sqrt{n}}{n}=0\]
		as desired.
	\end{proof}

	\section{Which Points are Accumulation Points?} \label{Cases}
	
	The aim of this section is to prove Theorem \ref{CasesThm}, describing exactly which points of $C$ are accumulation points of the locus $L =C\cap \bigcup_nY_0(n)$. First, to indicate how the arguments will proceed, the following explicit example will be studied.

	\subsection{Example of Non-Discreteness in Serre-Tate Coordinates} \label{Ex1}
	
	This calculation, which takes place in a simplified setting focussing only on the Serre-Tate coordinates side, does not use algebraicity of the curve in $j$-coordinates and assumes that the point in question is of a particularly nice form. It illustrates some of the ideas which will be generalised in the next section to find accumulation points of the locus $L=C\cap\bigcup_n Y_0(n)$ for $C$ any algebraic curve.
	
	Consider the curve $q_2 = q_1^3-2q_1^2+2q_1$ in an ordinary residue disc with centre $(\bar{E}^\text{can},\bar{E}^\text{can})$ with End$(\bar{E})=\zed[\sqrt{-d}]$, and intersections with tori $q_2 = q_1^{c_f}$ appearing in the image of \ $\bigcup_nY_0(n)$ in Serre-Tate coordinates. Setting $q_1 = 1+s$, have $q_1^3-2q_1^2+2q_1 = 1+s+s^2+s^3$, so the equation to solve is
	\[1+s+s^2+s^3 = 1+c_f s + \binom{c_f}{2} s^2 +....\]
	or
	\[(c_f-1)s + (\binom{c_f}{2}-1) s^2 + (\binom{c_f}{3}-1)s^3 + \binom{c_f}{4} s^4 +.... =0\]
	$s=0$ is always a simple root (for $c_f\ne1$). So to show non-discreteness, the first aim is to show that, for every closed ball $\bar{B}_m \colonequals \bar{B}_{1/p^m}(0)$, there are infinitely many $c_f$ for which the above equation has at least two solutions in $\bar{B}_m$. This is done in the Lemma below. Then since within the ball, distinct tori only intersect in points whose coordinates are roots of unity, and there are no roots of unity in the punctured ball $\bar{B}_{1/p^m}^*(1) =\{q_1\in 1+\mathfrak{m}_p \ | \ 0<|q_1-1|\leq \frac{1}{p^m} \}$ for $m\geq 1$, distinct numbers $c_f$ will yield different coordinates $q_1$ for their intersection. Lastly, $q_2 = (1+s_1)^{c_f}$ also lies in $\bar{B}_{1/p^m}(1)$ whenever $|s_1|\leq \frac{1}{p^m}$. So  this says that for each $m$, the set
	\[(q_2 = q_1^3-2q_1^2+2q_1)\cap \bar{B}_{1/p^m}^*((1,1))\cap \cup_{(n,p)=1}Y_0(n) \] is infinite. It remains to prove the following Lemma, a special case of the more general Corollary \ref{powercalculation} below
	
	\begin{lemma} \label{ExApprox}
		For each $m\geq 1$, there are infinitely many numbers $c_f = \frac{a-b\sqrt{-d}}{a+b\sqrt{-d}}$ with $a,b\in \zed$ and $(a^2+b^2d,p)=1$ such that the equation
		\[(c_f-1)s + (\binom{c_f}{2}-1) s^2 + (\binom{c_f}{3}-1)s^3 + \binom{c_f}{4} s^4 +.... =0\]
		has at least two solutions with $s\in \bar{B}_m = \bar{B}_{\frac{1}{p^m}}(0)$.
	\end{lemma}
	
	\begin{proof}
		
		For the ball $\bar{B}_m$, make the change of variable $s = p^mt$ to change the problem to solving, for $|t|\leq 1$,
		\[(c_f-1)p^mt + (\binom{c_f}{2}-1)p^{2m} t^2 + (\binom{c_f}{3}-1)p^{3m}t^3 + \binom{c_f}{4}p^{4m} t^4 +.... =0\]
		Consider numbers $c_f$ satisfying $c_f\equiv 1 \pmod{p^{m+1}}$. There are infinitely many such, for instance take $f[p^\infty]=1+p^{m+1}b\sqrt{-d}$ for $b\in \zed$. Then \[c_f = \frac{1-p^{m+1}b\sqrt{-d}}{1+p^{m+1}b\sqrt{-d}},\] are all distinct numbers satisfying $c_f\equiv 1 \pmod{p^{m+1}}$. The zero locus of the equation $q_2 = q_1^{c_f}$ is a branch of $Y_0(1+db^2p^{2m+2})$, where the degree $1+db^2p^{2m+2}$ is coprime with $p$.
		
		Then $\binom{c_f}{2} = \frac{c_f(c_f-1)}{2}$ is divisible by $p^{m}$ (even by $p^{m+1}$ if $p>2$), so $\binom{c_f}{2}-1\in -1+\mathfrak{m}_K$. Thus have
		\[v_p((c_f-1)p^m)\geq 2m+1, \qquad v_p\left(\left(\binom{c_f}{2}-1\right)p^{2m}\right)=2m\]
		Hence by Strassman's Theorem, there are at least two roots $t\in \oh_p$ of the above equation, so for every $c_f\equiv 1\pmod{p^{m+1}}$, there is a root of $q_1^3-2q_1^2+2q_1=q_1^{c_f}$ in the punctured ball $\bar{B}_{1/p^m}^*(1) $.	
	\end{proof}

	\subsection{Accumulation Points in General}

	If $C$ is a modular curve then it is clear that all points are accumulation points in the locus $L=C\cap \bigcup_nY_0(n)$. Thus in this section, the assumption that $C$ is not a Shimura subvariety of the product of modular curves remains. Since the questions in this section are local, it is not necessary to assume that the whole curve lies within the ordinary locus. These calculations work for a general curve $C$ around a point $P$ corresponding to a pair of elliptic curves with good ordinary reduction.
	
	Certainly if a point $P$ does not lie on $C$, it cannot be an accumulation point of a locus which is contained in $C$. So in the interest of studying accumulation points, begin by assuming that $P$ lies on $C$. In \S\ref{Ex1}, Lemma \ref{ExApprox} showed that the numbers $c_f$ could be used to well-approximate the value 1. The next Lemma is an analogue of this, with the aim of showing that the exponents appearing in the tori comprising the locus of modular curves in Serre-Tate coordinates approximate well all possible exponents.
	
	\begin{defn}\label{tori}
		A curve in Serre-Tate coordinates of the form $q_2 = q_1^c$ for $c\in \zed_p$ shall be referred to as a \textit{torus}. If further $c=c_{\bar{f}}=\frac{a_{\bar{f}}^T}{a_{\bar{f}}}$ for $\bar{f}$ an isogeny as in the construction in Section \ref{ModularCurves}, the curve shall be called a \textit{modular torus}. Also say that such a $c$ \textit{comes from a modular curve}.
	\end{defn}
	
	In this language, the following Lemma says that there are `enough' modular tori.
	
	\begin{lemma} \label{approxlemma}
		Let $\oh$ be an order of discriminant prime to $p$ in the ring of integers of the quadratic imaginary number field $\kew(\sqrt{-d})$, with $d>0$ a squarefree integer. Let $I\triangleleft \oh$ be a non-zero ideal. Then the collection of elements
		\[\{ \frac{\bar{z}}{z}\ |\ z\in I,\ v_p(z)\leq v_p(\bar{z})\}\] is dense in $\zed_p$. 
	\end{lemma}
	
	\begin{proof}
		
		Write $\oh = \zed\oplus e\omega \zed$ where $\{1,\omega\}$ form an integral basis for the ring of integers in $\kew(\sqrt{-d})$. The value of Disc$(\oh)$ is either $-e^2d$ or $-4e^2d$, so the hypotheses of the Lemma imply that $d$ and $e$ are coprime to $p$. Choose a non-zero element $\beta+\gamma\sqrt{-d}\in I$. After perhaps multiplying by 2, may assume $\beta,\gamma\in \zed$.

	All elements of the form $(u+v\sqrt{-d})(\beta+\gamma\sqrt{-d})$, where $u,v\in\zed$ and $e|v$, all lie in $I$ (although of course not every element need have this form). Thus the Lemma would follow from finding, for every $t\in \zed_p$ and for every sufficiently large $r>0$, a solution $u\in \zed$, $v\in e\zed$, to the congruence 
	\begin{align}\label{Case1Cong}
		t\equiv \frac{u\beta-vd\gamma-(v\beta+u\gamma)\sqrt{-d}}{u\beta-vd\gamma+(v\beta+u\gamma)\sqrt{-d}}\pmod{p^r},
	\end{align}
	The congruence is equivalent to 	
		\[
		(1-t)(u\beta-vd\gamma) \equiv (1+t)(v\beta+u\gamma)\sqrt{-d} \pmod{p^r},\] as long as both sides are not identically zero.	Rearranging further, the above becomes:		
		\begin{align*} v(\beta\sqrt{-d}(1+t)+d\gamma(1-t))\equiv u(\beta(1-t)-\gamma\sqrt{-d}(1+t))\pmod{p^r}.\end{align*}		
		Write $A = \beta\sqrt{-d}(1+t)+d\gamma(1-t)$ and $B= \beta(1-t)-\gamma\sqrt{-d}(1+t)$. If $A=B=0$, then computing the ratio $\frac{1+t}{1-t}$ (or its inverse if $t=1$) from each of $A=0$ and $B=0$ gives $\sqrt{-d}(\beta^2+d\gamma^2)=0$, a contradiction to having chosen a non-zero element of $I$ in the first place. Now, fix $r>0$.

		\textbf{Case 1:} Suppose first that $R = v_p(A) \leq v_p(B)$ (so $R<\infty$). Further, set \[S = v_p(N(\beta+\gamma\sqrt{-d}))<\infty.\] Then choose $(u,v)\equiv(1,\frac{B}{A})\pmod{p^{r}}$, where this makes sense because $\frac{B}{A}\in \zed_p$. Then consider \begin{align*} z &= (u+v\sqrt{-d})(\beta+\gamma\sqrt{-d})  \\ &\equiv \frac{1}{A}(A+B\sqrt{-d})(\beta+\gamma\sqrt{-d}) \pmod{p^r} \\ &\equiv \frac{2\sqrt{-d}(\beta-\gamma\sqrt{-d})(\beta+\gamma\sqrt{-d})}{A} \pmod{p^r}  \\ &\equiv \frac{2\sqrt{-d}N(\beta+\gamma\sqrt{-d})}{A}\pmod{p^r}.\end{align*} Thus if  $r>S$, $v_p(z) = S-R$. Further, \begin{align*} \bar{z} &= (u-v\sqrt{-d})(\beta-\gamma\sqrt{-d})  \\ &\equiv \frac{1}{A}(A-B\sqrt{-d})(\beta-\gamma\sqrt{-d}) \pmod{p^r} \\ &\equiv \frac{2t\sqrt{-d}(\beta+\gamma\sqrt{-d})(\beta-\gamma\sqrt{-d})}{A} \pmod{p^r} \\ &\equiv \frac{2t\sqrt{-d}N(\beta+\gamma\sqrt{-d})}{A}\pmod{p^r}.\end{align*} So $v_p(\bar{z}) \geq S-R$ (with equality if $2t$ is a unit). Then $\frac{\bar{z}}{z}\equiv t\pmod{p^r}$ as desired. Since $(e,p^r)=1$, it is possible to do this while simultaneously forcing $v\equiv 0\pmod{e}$ by the Chinese Remainder Theorem. 
		
		\textbf{Case 2:} If instead $v_p(B)\leq v_p(A)$, take $(u,v) \equiv (\frac{A}{B},1)\pmod{p^r}$ and $z = (u+v\sqrt{-d})(\beta+\gamma\sqrt{-d})$. Now compute
		 \[
		 z \equiv \frac{1}{B}(A+B\sqrt{-d})(\beta+\gamma\sqrt{-d}) \equiv \frac{2\sqrt{-d}N(\beta+\gamma\sqrt{-d})}{B}\pmod{p^r},
		\] so $v_p(z)=S-R$ and \[\bar{z}\equiv \frac{1}{B}(A-B\sqrt{-d})(\beta-\gamma\sqrt{-d}) \equiv \frac{2t\sqrt{-d}N(\beta+\gamma\sqrt{-d})}{B}\pmod{p^r},\] whence $v_p(\bar{z})\geq S-R$, and $\frac{\bar{z}}{z}\equiv t \pmod{p^r}$ with $v_p(z)\leq v_p(\bar{z})$. Again, by the Chinese Remainder Theorem, it is possible to simultaneously force $v\equiv 0\pmod{e}$. 
		\end{proof}

	The corresponding version of this statement for tori in Serre-Tate coordinates is the following. Let $\bar{E}_1,\bar{E}_2$ be isogenous ordinary elliptic curves over $\fpbar$. Fix an isogeny between them and use it to identify Hom$(\bar{E}_1,\bar{E}_2)$ with an ideal in an order in an imaginary quadratic field (as in the proof of Lemma \ref{listoftori}). Take Serre-Tate coordinates on the residue disc $B_R((\bar{E}_1$,$\bar{E}_2))$ (with $R<1$).

	\begin{cor} \label{powercalculation}
		The set \[\{c\in \zed_p \ | \ q_2 = q_1^c \text{ is a modular torus}\}\] is dense in $\zed_p$.

	\end{cor}
	
	\begin{proof}
		The order in question is End$(\bar{E}_2)$. An order in the quadratic field $\kew(\sqrt{-d})$ with $d>0$ a squarefree integer has the form $\zed\oplus e\omega \zed$ where $e\in \zed\setminus\{0\}$ and $\omega =  \sqrt{-d}$ or $\frac{1+\sqrt{-d}}{2}$. The discriminant of this order is $-4e^2d$ or $-e^2d$ respectively. A subring of this order is $\zed[$Frob] (Frob viewed as an element in the ring End$(E_2)$) which has discriminant $D_F = t^2-4q$. $D_F$ is coprime to $p$ by (one) definition of ordinariness. $D|D_F$ by containment, so $D$, and thus $e$ and $d$, are both coprime to $p$. Therefore the above Lemma can be applied. Identifying $z$ in the Lemma with the numbers $a_{\bar{f}}\in$ Hom$(\bar{E}_1,\bar{E}_2)$, the condition $v_p(z)\leq v_p(\bar{z})$ ensures that $q_2 = q_1^{\bar{z}/z}$ is a modular torus (and translates by $(p^{v_p(z)})^\text{th}$ roots of unity are also part of the locus $\bigcup_nY_0(n)$). So writing $c =\frac{\bar{z}}{z}$ gives an identification of the set \[\{c\in \zed_p \ | \ q_2 = q_1^c \text{ is a modular torus}\}\] with the dense set of the previous lemma.
	\end{proof}

	\subsubsection{Setting up the Equation for General Integral Points $P=(j_1^0,j_2^0)$}
	
	Suppose $P$ is a point on $C$ corresponding to a pair of elliptic curves which both have good ordinary reduction. The following two Lemmas regarding limits of points on different tori are required for all three of the cases in Theorem \ref{CasesThm}:

	\begin{lemma}\label{uniform}
		Let $a_n \in \mathfrak{m}_p$ with  $a_n \to a$ for some $a\in \mathfrak{m}_p$ and $c_n\in \zed_p$ with $c_n \to c$ for some $c\in \zed_p$. Then $(1+a_n)^{c_n}\to (1+a)^c$.
	\end{lemma}
	
	\begin{proof}
		Fix $A\in \enn$ and aim to find $N\in \enn$ such that, whenever $n\geq N$, $v_p((1+a)^c-(1+a_n)^{c_n})\geq A$. Rewriting and using the binomial formula, the expression to control is		
		\[(1-1) + (c_na_n-ca) + (\binom{c_n}{2}a_n^2-\binom{c}{2}a^2) + .... + (\binom{c_n}{j}a_n^j-\binom{c}{j}a^j) + ...\]
		Choose $J$ such that  $v_p(a^{J})\geq A$. Let $N_1$ be such that whenever $n\geq N_1$, $v_p(a_n)\geq \frac{1}{2}v_p(a)$. Then for $n\geq N_1$, $v_p(a_n^{2J})\geq A$. Since binomial coefficients always lie in $\zed_p$, whenever $n\geq N_1$ and $j\geq J$, have $v_p(\binom{c_n}{j}a_n^j-\binom{c}{j}a^j)\geq A$. Therefore it suffices to focus on the terms with powers up to $J-1$.
		
		Let $B = \max_{j\leq J} v_p(j!)$. Choose $N_2$ such that for $n\geq N_2$, $v_p(c_n-c)\geq A+B$ and $v_p(a_n-a)\geq A+B$. Then for every $j\leq J$,	
		\[v_p\left(\binom{c_n}{j}a_n^j-\binom{c}{j}a^j\right)=v_p\left(\frac{c_n(c_n-1)...(c_n-j+1)a_n^j-c(c-1)...(c-j+1)a^j}{j!}\right) \geq A\]		
		since the two terms in the numerator are equal modulo $p^{A+B}$. Letting $N=\max\{N_1,N_2\}$ achieves the aim.
	\end{proof}

	\begin{lemma}\label{MustTorus}
		Let $P$ be an integral point on $C$ which is an accumulation point of the locus $L=C\cap \bigcup_nY_0(n)$. Then $P$ must lie on a torus translate.
	\end{lemma}
	
	\begin{proof}	
		Suppose $P =(q_1^0,q_2^0)$ (in Serre-Tate coordinates) is an accumulation point of $L$. With notation as above, that would mean that there exist $q_1^{(n)}\to q_1^0$, $q_2^{(n)}\to q_2^0$, $\zeta_n$ $p$-power roots of unity and $a_n,b_n\in \zed_p$ such that $q_2^{b_n} = q_1^{a_n}$ are modular tori, $(q_2^{(n)})^{b_n}= (q_1^{(n)})^{a_n}$ and $(q_1^{(n)},q_2^{(n)})$ lies on $C$. Infinitely many must satisfy either $v_p(a_n)\geq v_p(b_n)$ or $v_p(a_n)\leq v_p(b_n)$, so after taking a subsequence and relabelling, assume the former. Then setting $c_n = \frac{a_n}{b_n}$, there exist  $p$-power roots of unity $\zeta_n$ such that $q_2^{(n)}= \zeta_n(q_1^{(n)})^{c_n}$.
		
		Since $\zed_p$ is compact, after taking a further subsequence, may assume $c_n\to c\in \zed_p$. Taking the limit in $n$, $q_2^0=\lim_{n\to\infty}\zeta_n(q_1^{(n)})^{c_n}$. Then, by Lemma \ref{uniform}, $(q_1^{(n)})^{c_n}\to (q_1^0)^c\ne0$ so the sequence $\zeta_n$ converges too. Since roots of unity are isolated, the sequence $\zeta_n$ is eventually constant, say equal to $\zeta$. Thus $q_2^0 = \zeta(q_1^0)^c$. So to be an accumulation point of $L$, $P$ must lie on a torus translate (though not necessarily a modular torus translate).
	\end{proof}

	For the question of whether $P=(q_1^0,q_2^0)$ is an accumulation point of the locus $L$, consider intersections of the torus $q_2= \zeta q_1^c$ with $C$ near $P$. Write $C=\mathbb{V}(F)$ for $F\in \oh_p[J_1,J_2]$. First take the coordinates $s_1 = q_1 -1$ and $s_2 = q_2-\zeta$, both running over $B_1(0)$. Using the change of coordinates to Serre-Tate coordinates as in \S\ref{FinishProof}, convert the equation $F(j_1,j_2)=0$ into $s$-coordinates and then substitute in $s_2 = \zeta(1+s_1)^c-\zeta$ to yield a power series in just the variable $s_1$. Lastly, recentring at $q_1^0-1$ (that is, letting $w_1 = q_1-q_1^0$) gives a one variable power series \[\sum_{i\geq0} \beta_i(\zeta,c)w_1^i,\] with coefficients $\beta_i(\zeta,c)$ continuous in $c$ and satisfying $v_p(\beta_i(\zeta,c))\geq 0$ (recall $P$ was an integral point on $C$ since the corresponding elliptic curves had good reduction). If $w_1$ is a zero of this power series then the point with $q$-coordinates $(q_1^0+w_1,\zeta(q_1^0+w_1)^c)$ lies on $C$.

	To show that $P$ is an accumulation point of $L$ is to find a sequence $(c_n)$ coming from modular curves and $\zeta_n$ roots of unity such that there are solutions to equations of the form
	\[\sum_{i\geq0} \beta_i(\zeta_n,c_n)w_1^i=0.\] with $|w_1|$ positive but arbitrarily small. Now split into cases according to whether or not the coordinates of the point $P$ correspond to CM lifts. Throughout, $K$ will denote a finite extension of $W(\fpbar)[1/p]$.

	\subsubsection{Neither Coordinate is a CM $j$-invariant}
	
	\begin{prop} \label{nonCM}
		Let $P=(j_1^0,j_2^0)$ be an integral point on $C$ with neither coordinate a CM $j$-invariant. Then $P$ is an accumulation point of $L$ if and only if $C$ lies on a torus in Serre-Tate coordinates.
	\end{prop}

	\begin{proof}
		As with the notation above, let $(q_1^0,q_2^0)$ be the Serre-Tate coordinates of $P$ and write $s_i^0 = q_i^0-1$.
		
		The only $\zed_p$ subtorus translates on which $P$ could lie corresponds to $c = \frac{\log({1+s_2^0})}{\log({1+s_1^0})}$, (or its inverse depending on valuation). If this ratio is not defined over $\zed_p$ then there is no chance of being an accumulation point of $L$ by Lemma \ref{MustTorus}.
		
		On the other hand, to show that any $P$ which lies on a torus is an accumulation point of $L$, suppose this value $c$ does lie in $\zed_p$. Then (after possibly swapping the coordinates) $q_2^0 = \zeta(q_1^0)^c$ for some $p$-power root of unity $\zeta$. Fix $m$ and try to find an intersection point within a distance $1/p^m$ of $P$. The assumption that both coordinates are non-CM means $c\ne0$. Since $C$ is not a Shimura subvariety, there exists $j\geq1$ with $\beta_j(\zeta,c)\ne0$ (in the notation set up after Lemma \ref{MustTorus}). Now take a sequence $c_n\to c$ with $c_n\in \zed_p\setminus\{c\}$ coming from modular curves (which is possible by Corollary \ref{powercalculation}, as such are dense and $\zed_p$ has no isolated points). Then, as the coefficients are continuous in $c$, $\lim_{n\to\infty}\beta_j(\zeta,c_n)= \beta_j(\zeta,c)\ne0 $. Therefore (after perhaps throwing away early terms) the sequence $(\beta_j(\zeta,c_n))_n$ is bounded from below in absolute value uniformly in $n$, by $\frac{1}{p^a}$, say. Then given $m$, choose $n$ such that $jm+a<v_p(\beta_0(\zeta,c_n))<\infty$ (the left hand inequality is possible because $c_n\to c$ and $\beta_0(\zeta,c)=0$, and the right hand follows since $c_n\ne c$ and $P$ lies on only one torus translate, $q_2= \zeta q_1^c$). Thus $|\beta_0(\zeta,c_n)|<|\beta_j(\zeta,c_n)||p^{mj}|$, so Strassman's Theorem then implies that there is a solution $x=x_n$ to the equation 		
		\[\beta_0(\zeta,c_n)+\beta_1(\zeta,c_n)p^mx + \beta_2(\zeta,c_n)p^{2m}x^2 + ... +\beta_j(\zeta,c_n)p^{jm}x^j+... \]		
		with $0<|x_n|<1$ ($x=0$ is not a solution as $P$ is not on any torus translate with exponent $c_n\ne c$). That is, there is an intersection of $C$ with the torus translate by $\zeta$ corresponding to $c_n$ in the punctured disc of radius $1/p^m$ around $P$ as required. (Note that so far, this only gives that the first coordinate is close to that of $P$, but if $s_1^{(n)}\colonequals p^nx_n$, then as $c_n\to c$ and $s_1^{(n)}\to s_1^0$, it follows from Lemma \ref{uniform} that the second coordinates also converge: $(1+s_1^{(n)})^{c_n}\to (1+s_1^0)^c= q_2^0$).
	\end{proof}

	\subsubsection{Exactly One Coordinate is a CM $j$-invariant}
	
	\begin{prop}
		Let $P= (j_1^0,j_2^0)$ be an integral point on $C$ with only one coordinate a CM $j$-invariant. Then $P$ is not an accumulation point of the locus $L=C\cap \bigcup_nY_0(n)$.
	\end{prop}
	
	\begin{proof}
		The equality $q_1^a = q_2^b$ can never be satisfied for $a,b\in \zed_p$ if exactly one of $q_1$ or $q_2$ is a root of unity, since no $\zed_p$-power of a number that is not a root of unity can be a root of unity. So $P$ does not lie on a torus, whence $P$ cannot be an accumulation point of $L$ by Lemma \ref{MustTorus}.
	\end{proof}

	\subsubsection{Both Coordinates are CM $j$-invariants}\label{CMcase}
	
	Suppose $(j_1^0,j_2^0)$ is a point on $C$ with both coordinates CM lifts. Let $(q_1^0,q_2^0)$ be the corresponding Serre-Tate coordinates, which are both roots of unity. Let $a\geq0$ be such that $q_1^0$ is a primitive $(p^a)^\text{th}$ root of unity, and suppose (after perhaps swapping coordinates) that $P$ lies on the torus translate $q_2 = \zeta q_1^c$.
	
	\begin{prop}
		 A point $P$ which lies on the torus translate $q_2 = \zeta q_1^c$ is an accumulation point of $L=C\cap \bigcup_nY_0(n)$ if and only if $C$ is tangent at $P$ to a torus translate $q_2= \zeta q_1^{c'}$ with $c'\equiv c \pmod{p^a}$.
	\end{prop}

	\begin{proof}
		
		Since $j_1^0$ is a CM $j$-invariant, $q_1^0$ is a primitive $(p^a)^\text{th}$ root of unity for some $a\geq0$. Thus $q_2^0 = \zeta (q_1^0)^\text{c'}$ whenever $c'\equiv c\pmod{p^a}$. In particular, $P$ lies on infinitely many modular torus translates. 
		
		Firstly, observe that any accumulation point of $L$ can be realised by considering only points on torus translates containing $P$ itself. Indeed, given sequences $q_1^{(n)}\to q_1^0$, $q_2^{(n)}\to q_2^0$, $\zeta_n$ $p$-power roots of unity and $c_n\in \zed_p$ such that 
		\begin{align}\label{passthroughP}
			 q_2^{(n)}=\zeta_n(q_1^{(n)})^{c_n}
		\end{align}
		 for each $n$, replace $c_n$ by a convergent subsequence which in particular must eventually lie in only one congruence class modulo $p^a$, say $d+p^a\zed_p$. Then just as in the proof of Lemma \ref{MustTorus} (and using Corollary \ref{uniform} again), the $\zeta_n$ are eventually constant, say equal to $\zeta'$ and taking the limit in (\ref{passthroughP}), obtain $q_2^0 = \zeta'(q_1^0)^{d'}$ for some $d'\in d+p^a\zed_p$.  But now for $n$ large enough that $c_n\in d+p^a\zed_p$, $(q_1^0)^{d'}=(q_1^0)^{c_n}$. So $(q_1^0,q_2^0)$ lies on the torus $q_2 = \zeta'q_1^{c_n}$. Thus to study accumulation points of $L$ it suffices to consider only torus translates passing through $P$.

		Suppose first that $P$ is an accumulation point of $L$. Then for every $m>0$ there exist $\zeta_m$ a root of unity, $c_m\in c+p^a\zed_p$ coming from modular curves and $j_m>1$ such that $v_p(\beta_1(\zeta_m,c_m)) \geq v_p(\beta_{j_m}(\zeta_m, c_m))+mj_m$. (As above, discreteness of roots of unity means $\beta_0(\zeta_m,c_m)$ has to be 0 - i.e. $P$ lies on the corresponding torus translate - when $m$ is sufficiently large). After taking a subsequence, $c_m$ has a limit, $c'\in c+p^a\zed_p$, say. Then since everything is integral, $\lim_{m\to\infty}v_p(\beta_1(\zeta_m,c_m))=\infty$. Since there are only finitely many possible $\zeta_m$ which can appear (the order of $\zeta_m$ is bounded by the maximum order of $q_1^0$ and $q_2^0$), one, $\zeta'$, say, must appear infinitely often. So taking the corresponding subsequence of the exponents, may assume $\beta_1(\zeta',c')=0$. Thus $P$ is a repeated root of the equation of intersection of $C$ and the torus translate $q_2 = \zeta' q_1^{c'}$. Thus $C$ is tangent to this torus translate at $P$.

		Conversely, suppose $C$ is tangent at $P$ to the torus translate $q_2 = \zeta'q_1^{c'}$ with $c'\in c+p^a\zed_p$. First note that $C$ can be tangent at $P$ to at most one modular curve with exponent congruent to $c$ modulo $p^a$, since if $\zeta_1 c_1(1+s_1^0)^{c_1-1}=\zeta_2 c_2(1+s_1^0)^{c_2-1}$ with $c_1,c_2\in c+p^a\zed_p$, then $(1+s_1^0)^{c_1-1}=(1+s_1^0)^{c_2-1}\ne0$, meaning that $\zeta_1c_1=\zeta_2c_2$. Then $\zeta_1/\zeta_2=c_2/c_1$ is a $p$-power root of unity in $\zed_p$, hence equal to $1$, and so $c_1 =c_2$ too.

		Then by Corollary \ref{powercalculation}, there exist $c_n\ne c'$ coming from modular curves with $c_n\to c'$. Since $C$ is not a Shimura subvariety, there is $j$ minimal with $\beta_j(\zeta',c')\ne0$. Then since $c_n\to c'$, $\beta_1(\zeta',c_n)\to \beta_1(\zeta',c')=0$ but $v_p(\beta_j(\zeta',c_n))$ is bounded as $n$ ranges, say by a bound $A\in \enn$. Thus to find a second root in the ball of radius $1/p^m$ around $P$, if $N$ is such that for all $n>N$, $|\beta_1(\zeta',c_n)|<1/p^{mj+A}$, have $|\beta_1(\zeta',c_n)||p^m|<|\beta_j(\zeta',c')||p^{mj}|$. Now recalling that $\beta_0(\zeta',c')=0$, by Strassman's Theorem there will be at least two points of intersection (counted with multiplicity) in the ball of radius $1/p^m$ at $P$.	Since $c_n\ne c'$ for any $n$, the root at the centre is simple in each case, so there will be at least one point of the locus $C\cap \bigcup_nY_0(n)$ in every punctured ball around $P$.	
	\end{proof}
	
	Now that a complete characterisation of the accumulation points has been obtained, it is clear that for any curve $C$, there will be an accumulation point of $L$.
	
	\begin{cor} \label{always}
		For any curve $C$ defined over \cp, the locus $L =C\cap \bigcup_nY_0(n)$ has an accumulation point.
	\end{cor}
	
	\begin{proof}
		$C$ contains only finitely many points where both coordinates have complex multiplication by \cite{And98}. But the locus $L$ is infinite, so there must contain a point with non-CM coordinates.
		Further, infinitely many intersection points are in ordinary residue discs, by \cite{CO06}, Proposition 7.3. Every point in $L$ satisfying these two conditions is on a torus in Serre-Tate coordinates (in particular, on a modular torus), so is an accumulation point by Proposition \ref{nonCM}.
	\end{proof}

	\section{Relation to Previous Work} \label{Relate}
	
	In \cite{MP09}, as a special case of their main result on jumping loci, Davesh Maulik and Bjorn Poonen obtain the following Proposition. Here, $K$ is again a finite extension of $W(\fpbar)[1/p]$. In this section, $C$ is assumed to avoid points where both coordinates have supersingular reduction in order to make use of the calculations of previous sections. This is not a requirement in \cite{MP09}.
	
	\begin{prop}\label{MP} (Proposition 1.13 in \cite{MP09})
		Let $C$ be an irreducible curve defined over \ok, and $\mathfrak{X}\to C$ be an Abelian scheme over $C$. Then the set		
		\[\{c\in C(\oh_p) \ | \ \End(\mathfrak{X}_{\bar{\eta}})\to \End(\mathfrak{X}_c) \text{ is not an isomorphism}\}\] is $p$-adically nowhere dense.
	\end{prop}
	
	The specific family which has been considered in previous sections is where $C$ is a plane curve and the fibre above the point $(j_1,j_2)$ on $C$ is the product $E_{j_1}\times E_{j_2}$ of elliptic curves with $j$-invariants $j_1$ and $j_2$ respectively.
	
	The methods here recover the following version of Proposition \ref{MP} for this family:
	
	\begin{prop}
		Let $C$ be an irreducible curve in $X(1)\times X(1)$ defined over \ok \ and avoiding points where both coordinates have supersingular reduction. Suppose $C$ is not a Shimura subvariety. Then the locus of points 
		\[\{(j_1,j_2)\in C(\oh_p) \ | \ j_1 \text{ or } j_2 \text{ is CM, or } (j_1,j_2)\in \bigcup_nY_0(n)(\oh_p) \}\] is $p$-adically nowhere dense.
	\end{prop}
	
	To recover this result from the previous sections, first observe that in any ball of radius less than 1, the points which have at least one CM coordinate lie on a line where one of the corresponding Serre-Tate coordinates is a root of unity. In a ball of radius less than 1, there are only finitely many such roots of unity, hence the number of points on $C$ in any such ball is finite. Thus it suffices to show that the locus $C\cap\bigcup Y_0(n)$ is $p$-adically nowhere dense. 
	
	\begin{prop}
		The locus $L=C\cap \bigcup Y_0(n)$ is $p$-adically nowhere dense in $C$, provided $C$ is not a Shimura subvariety of the product of modular curves.
	\end{prop}
	
	\begin{proof}
		Suppose this locus were dense in some ball $B = B_R((a_1,a_2))\cap C$. This means that every point in $B$ is an accumulation point of $L$. After perhaps shrinking $B$, it is possible to assume that $R<1$ and that all points in $B$ have both coordinates non-CM $j$-invariants, by discreteness. Then by Proposition \ref{nonCM}, every point in $B$ lies on a torus. After changing to Serre-Tate coordinates, this means that there is an analytic map $f\colon B\to \zed_p$ sending $P = (q_1,q_2)$ to $\frac{\log(q_2)}{\log(q_1)}$. By the Open Mapping Theorem (see Lemma \ref{open} below), $f$ must be constant, so this component of $C$ is a torus translate on the whole residue disc by the identity principle.
		
		That is, $C$ is formally linear and algebraic and contains a CM point (it passes through a point at the centre of the disc of the form $(1,\zeta)$ with roots of unity as coordinates). Thus it follows from Theorem 4.5 in \cite{Moo98II} that $C$ is a Shimura subvariety (or modular torus translate, in the language of Definition \ref{tori}).
	\end{proof}
	
	Certainly the Open Mapping Theorem for such functions must be known, but a proof is summarised below:

	\begin{lemma}\label{open}
		Let $U\subset \cp$ be open and  $f\colon U\to \cp$ a non-constant analytic map. Then $f$ is an open map.
	\end{lemma}
	
	\begin{proof}
		The argument is slightly different to the usual argument over $\see$, which uses local compactness. However it still hinges on a $p$-adic version of Rouch\'e's Theorem.
		
		Let $y_0\in f(U)$ and choose a preimage $x_0\in U$. The function $f(x)-y_0$ is non-zero and analytic, so changing coordinates to centre at $x_0$ and rescaling for convergence on $|t|\geq 1$ (that is to say, choose $a>0$ such that, if $p^at =  x-x_0$, then  $|\frac{x-x_0}{p^{a}}|\leq1$ implies $x\in U$), can expand
		\[f(x)-y_0 = a_1t + a_2t^2 + ....\]
		Let $r = \max_n|a_n|$ (this is $||f||_1$ in the notation of Proposition \ref{epsdel}), and $N$ be the position in which this maximum is achieved (convergence on $|t|\leq1$ implies $|a_n|\to0$.) Now for any $y\in B_r(y_0)$, apply $p$-adic Rouch\'e's Theorem (Theorem 10.10 in \cite{Con}) to the analytic functions $g(x) =f(x)-y_0$ and $h(x)=f(x)-y$. $||g-h||_1 = |y-y_0|<r$ so if $h(t) = b_0+b_1t+b_2 t^2+...$, the maximum of the numbers $\{|b_n|\}$ is also achieved at $n=N\geq 1$. Thus by Strassman's Theorem, $h$ has a root with $|t|\leq 1$, which exactly means $y\in f(U)$.		
	\end{proof}

	\section{The Supersingular Locus} \label{SecSsing}

	\subsection{Background on the Supersingular Period Map} \label{Period}
	
	Let $E/\cp$ be an elliptic curve with supersingular reduction and consider the vector bundle $\mathcal{H}^1_\text{dR}$ on the deformation disc (open ball of radius 1) containing $E$. Let $\{v_1,v_2\}$ be a basis of flat sections for this bundle locally, with Fil$^1$ spanned by $v_1$ at $E$. Solving the Gauss-Manin connection on a neighbourhood of $E$ on the modular curve gives a way to identify fibres of the vector bundle $\mathcal{H}^1_\text{dR}$ and track how the filtration at other points in the disc looks in terms of a flat basis. The period map is given by sending an elliptic curve $E'$ within the region on which the connection can be solved to the line Fil$^1(E')$ in the basis $\{v_1,v_2\}$. 
		
	Observe now that this can be done on the whole ball: indeed to solve the connection is to solve the following differential equation. Suppose 
	\[\nabla v_1 = av_1+cv_2, \qquad \nabla v_2 = bv_1+dv_2\]
	where $a,b,c,d\in \Omega^1_{Y_0(1)/\oh_p}$. Then $f_1v_1+f_2v_2$ is flat if and only if
	\[d(f_1) = -af_1-bf_2, \qquad d(f_2) = -cf_1-df_2\]
	
	Now solving in formal power series centred at $E$, suppose $f_1 = \sum_n k_nz^n$ and $f_2 = \sum_nl_nz^n$, with $k_0 = 1$ and $l_0 = 0$. Then these equations become 
	\[k_{n+1} = \frac{-1}{n+1} (ak_n+bl_n), \qquad l_{n+1} = \frac{-1}{n+1} (ck_n+dl_n)\]
	
	That is (since $a,b,c,d$ take integral values) the denominators of $k_{n+1}$ and $l_{n+1}$ can be as bad as $(n+1)!$, so there is convergence for $|z|< p^\frac{1}{p-1}$. However, there is more structure since $\mathcal{H}^1_\text{dR}$ is isomorphic to the vector bundle $H^1_\text{crys}$, which has the structure of an $F$-isocrystal. Explicitly this means that, considering the contraction map given on rings by
	\[f\colon W[[x]]\to W[[x]], \qquad x\mapsto x^p, \qquad a\mapsto \sigma(a) \text{ for } a\in W\]	
	there is an isomorphism of vector bundles with Frobenius and connection between $f^*H^1_\text{crys}[1/p]$ and $H^1_\text{crys}[1/p]$. Therefore any supersingular ball of radius $R<1$ on $Y(1)$ can be filled out, covered with flat coordinates - and thus given a period map to $\mathbb{P}^1$ - by sufficiently many pullbacks of the original domain $B_{1/p^{1/p-1}}$, on which the differential equation above was directly soluble. Since $p$ was inverted in the isomorphism between $f^*H^1_\text{crys}[1/p]$ and $H^1_\text{crys}[1/p]$, this tiling process may introduce more denominators, but for any ball of radius $R<1$ this will only be a bounded change.
	
	\subsection{Properties of Hecke Orbits} \label{Hecke}
	
		The upshot of the previous section is that, when studying a 2-dimensional supersingular ball of radius $R<1$ in $Y(1)\times Y(1)$ in the context of (\ref{ssingsubseq}), the period map will always be available, a map of rigid analytic spaces from the closure of the ball to $\mathbb{P}^1\times \mathbb{P}^1$, coordinatewise given by the above. Implicitly fixing a centre $(a_1,a_2)$, call this map $\pi_{R}\colon\bar{B}_R((a_1,a_2))\to \mathbb{P}^1$. 
		
		The following results about Hecke orbits, which will be used to compute the limit (\ref{ssingsubseq}), come from work of Michael Rapoport and Thomas Zink on moduli of $p$-divisible groups.

	\begin{thm}\label{ppower}
		Fix $R<1$, a centre $(a_1,a_2)\in Y(1) \times Y(1)$ with supersingular coordinates and a positive integer $n_0$ coprime to $p$. Then the image $\pi_R(Y_0(n_0p^m))$ is independent of $m$. 
	\end{thm}
	
	\begin{proof}
		This follows from a stronger statement about when two points have the same period image in \cite{RZ96}, Proposition 5.37. 
	\end{proof}

	\begin{cor} \label{boundedintersections}
		Fix $R<1$, a centre $(a_1,a_2)\in Y(1) \times Y(1)$ with supersingular coordinates and a positive integer $n_0$ coprime to $p$. There is $M\in \enn$ such that \[m\geq M \quad \Longrightarrow\quad Y_0(n_0p^m)\cap \bar{B}_R((a_1,a_2))=\emptyset.\]
	\end{cor}
	
	\begin{proof}
		By Theorem \ref{ppower}, $\bigcup_m Y_0(n_0p^m)\cap \bar{B}_R((a_1,a_2)) \subset \pi_R^{-1}(\pi_R(Y_0(n_0)))$. $\pi_R(Y_0(n_0))$ is a closed rigid analytic subvariety of the target, thus its preimage can contain only finitely many components. That is, for only finitely many $m$ does $Y_0(n_0p^m)$ intersect $\bar{B}_R((a_1,a_2))$.
	\end{proof}

	\subsection{Applications to Equidistribution}

	Let $\gamma_1,..,\gamma_s$ be a list of the supersingular elliptic curves in characteristic $p$ and $\Gamma_1,...,\Gamma_s$ be lifts to unramified extensions of $\kew_p$ (which exist by Deuring's Lifting Theorem). For each $(i,j) \in \{1,...,s\}^2$, and each $N\in \enn$, write $R_N =\frac{1}{p^{1/N}}$ and  consider the ball $B_{R_N}((\Gamma_i,\Gamma_j))$. These exhaust all balls appearing in (\ref{ssingsubseq}) - every ball for which the limit there needs to be computed will be contained in one of these.
	
	Fix $N$ and write $R_N = \frac{1}{p^{1/N}}$. Write $L_N = \bigcup_{i,j} \bar{B}_{R_N}((\Gamma_i,\Gamma_j))$. Then by Corollary \ref{boundedintersections} with $n_0=1$, for each $i,j$, there exists $M_{ij}$ such that if $m\geq M_{ij}$, $Y_0(p^m)\cap \bar{B}_{R_N}((\Gamma_i,\Gamma_j))=\emptyset$. Thus if $M \colonequals \max_{i,j}M_{ij}$, for $m\geq M$, the intersection $Y_0(p^m)\cap L_N =\emptyset$.

	\begin{lemma} \label{primetop}
		Fix $N$ and define $M$ as above. Let $n_0$ be prime to $p$. Then if $m\geq M$, $Y_0(n_0p^m)\cap L_N =\emptyset$.
	\end{lemma}
	
	\begin{proof}
		Suppose $Y_0(n_0p^m)\cap L_N \neq\emptyset$, and pick $(E_1,E_2)$ in the intersection. From the definition of $L_N$, let $i$ and $j$ be such that $E_1 \equiv \Gamma_i$ and $E_2 \equiv \Gamma_j$ modulo $p^{1/N}$. 
		
		Factor the $n_0p^m$ isogeny from $E_1$ to $E_2$ as $E_1 \to E'\to E_2$, where the first part has degree $p^m$ and the second has degree $n_0$. Then $E'$ and $E_2$ are prime-to-$p$ isogenous. Let $E_2\to E'$ be an $n_0$-isogeny (for instance, the dual of the one above) with kernel $C\leq E_2[n_0]$ a cyclic subgroup of order $n_0$. Let $\tilde{C}\leq \Gamma_j[n_0]$ be the subgroup lifting the image of $C$ under the isomorphism $E_2\times \oh_p/p^{1/N}\oh_p \to \Gamma_j \times \oh_p/p^{1/N}\oh_p$ (this image can be lifted back up to $\Gamma_j$ since prime-to-$p$-torsion group schemes are \'etale). Then $\Gamma_j/\tilde{C}$ is defined over an unramified extension of the field of definition of $\Gamma_j$ (from \cite{Sil86}, Proposition VII.4.1 and Remark III.4.13.2), so in particular, it is defined over $W[1/p]$. 
		
		That is, $E' = E_2/C$ and $\Gamma_j/\tilde{C}$ are isomorphic modulo $p^{1/N}$. Letting $k$ be such that $\Gamma_j/\tilde{C}\cong \Gamma_k \pmod{\mathfrak{m}_p}$, it must in fact be the case that this isomorphism holds modulo $p$ since both are defined over an unramified base. Thus recover that $E'$ and $\Gamma_k$ are isomorphic modulo $p^{1/N}$ and so $(E_1,E')\in Y_0(p^m)\cap \bar{B}_{R_N}((\Gamma_i,\Gamma_k))\subset Y_0(p^m)\cap L_N$. Hence $m< M$ as required.
	\end{proof}	
	This is the final ingredient required to prove Theorem \ref{Main} (ii)

	\begin{thm}\label{ThmSsing}
		The limit in (\ref{ssingsubseq}) is 0.
	\end{thm}
	
	\begin{proof}
		Let $a_n$ be a sequence of integers with $v_p(a_n)\to \infty$, as in the Theorem statement, and let $B$ be a ball of radius less than 1 inside a supersingular residue disc. There exists $N$ and $i,j\in \{1,...,s\}^2$ such that $B\subset B_{R_N}((\Gamma_i,\Gamma_j))$ with notation as above, so it suffices to evaluate the limit for such balls. But Lemma \ref{primetop} says that the numerator in (\ref{ssingsubseq}) is zero whenever $v_p(a_n)\geq M$ ($M$ depended on $N$, which is fixed here), since the intersection $Y_0(a_n)\cap B=\emptyset$ in this case. Hence the limit is zero as desired. 
	\end{proof}

	\subsection{Non-equidistribution For a Sequence $a_n$ Coprime to $p$}
	
	In this subsection, let $C$ be a curve smooth over $\zed_p$ and not reducing to any modular curve modulo $p$. Let $\gamma$ be a supersingular elliptic curve in characteristic $p$ and $\Gamma$ be a lift to an unramified base as above. Let $a_n\to \infty$ be a sequence of prime numbers distinct from $p$. The modular curves $Y_0(a_n)$ are also smooth over $\zed_p$. $\gamma$ has as many self-isogenies of degree $a_n$ as the size of the set
	\[\{(a,b,c,d)\in \zed^4 \ | \ a^2+b^2+pc^2+pd^2=a_n\},\]
	
	which has asymptotic size linear in $a_n$. Then $(\gamma,\gamma)$ lies on each $Y_0(a_n)$. The intersection of $C$ with $Y_0(a_n)$ here is an intersection of smooth curves, so intersection points can be uniquely lifted to characteristic 0 unramified fields by Hensel's Lemma. Therefore the number of intersection points in $\bar{B}_{1/p}$ is at least linear in $a_n$, whence the limiting behaviour of the ratio in (\ref{ssingsubseq}) satisfies
	\[\limsup_n \frac{|C\cap Y_0(a_n)\cap B_{1/p^{1/2}}|}{\text{deg}{C}\text{ deg}Y_0(a_n)}>0.\] Hence equidistribution fails, using the converse of Lemma \ref{padicconditionproof} for $C$, similarly to the calculation in Remark \ref{CModular}.

	This behaviour is believed to be typical for curves (at least defined over discretely valued fields) and sequences $a_n$ with bounded $p$-divisibility.

	\section{Around Infinity}\label{infinity}
	
	To complete the proof of Theorem \ref{Main} using the strategy of Lemma \ref{padiccondition}, after the work of Sections \ref{SecOrd} (ordinary residue discs) and \ref{SecSsing} (supersingular residue discs), it remains to compute the limit (\ref{Limit}) for balls of the form $B=B^\infty_R((a_1,a_2))$.
	
	Note first that no modular curve passes through a point of the form $(a,\infty)$ or $(\infty,b)$ where $a,b\in \oh_p$, since at infinity, the reduction modulo $p$ is isomorphic to $\mathbb{G}_m$, whereas when the $j$-invariant is integral, there is good reduction. Thus, after perhaps enlarging the radius to a value still less than 1, the only balls that need to be considered at infinity are those containing $(\infty,\infty)$. To study this, Tate's coordinate from the $p$-adic uniformisation of elliptic curves will be used. This is reviewed below, with material coming from \cite{Sil94}.

	\begin{lemma}
		Fix $n\in \enn$. In Tate coordinates, the locus $\bigcup_{n'|n}Y_0(n')$ is cut out by the equations of the form $q_2^a=q_1^b$ where $a,b$ are divisors of $n$ with $ab=n$.
	\end{lemma}
	
	\begin{proof}
		Given a curve $K^*/q_1^\zed$, the points on the modular curve with first coordinate $q_1$ correspond to taking a quotient of $K^*/q_1^\zed$ by a cyclic subgroup of order $n$. The $n$-torsion is generated by a choice of an $n^\text{th}$ root of $q_1$ and the primitive $n^\text{th}$ roots of unity.
		
		For a subgroup generated by a choice $q_1^{1/n}$ of an $n^\text{th}$ root of $q_1$, have quotients of the form $(K^*/q_1^\zed)/(q_1^{1/n})^\zed = K^*/(q_1^{1/n})^\zed$. That is, points $(q_1,q_2)$ on the curve $q_2^n=q_1$ lie on $Y_0(n)$.
		
		In the case of the unique subgroup generated by $n^\text{th}$ roots of 1, $(K^*/q_1^\zed)/\langle \zeta_n\rangle\cong (K^*/(q_1^n)^\zed)$ (the map from $K^*/q_1^\zed$ being the $n^\text{th}$ power map), which gives, for each value of $q_1$, one more point on $Y_0(n)$, which is $(q_1,q_2^n)$.

		These are special cases of the general case of computing $(K^*/q_1^\zed)/ (\zeta_mq_1^\frac{1}{a})^\zed$ for $a$ and $m$ divisors of $n$ such that lcm$(a,m)=n$ (the lcm condition guaranteeing that such an element generates a subgroup of order exactly $n$). Taking $a=n$ with varying $m$ recovers the first case, and $a=1$, $m=n$ recovers the second. The map
		\[K^*/q_1^\zed\to K^*/(q_1^{n/a})^\zed\to K^*/(\zeta^{n/a}_m q_1^{n/a^2})\]
		has kernel $(\zeta_mq_1^{1/a})^\zed/q_1^\zed$. Thus the point $(q_1,q_2) = (q_1,\zeta_m^{n/a}q_1^{n/a^2})$ lies on a curve of the form $q_2^a=q_1^\frac{n}{a}$ for a divisor $a$ of $n$. 
		
		Conversely, given $(q_1,q_2)$ satisfying $q_2^a = q_1^{n/a}$, taking $a^\text{th}$ roots gives $q_2 = \zeta_c q_1^{n/a^2}$ where $\zeta_c$ is a primitive $c^\text{th}$ root of unity for some $c|a$. Now $m =\frac{cn}{a}$ is a divisor of $n$, so can choose a primitive $m^\text{th}$ root of unity $\zeta_m$ satisfying $\zeta_m^{n/a}=\zeta_c$. Then the above map can be constructed with kernel $(\zeta_mq_1^{1/a})^\zed$ for some primitive root $\zeta_m$. This cyclic subrgoup has size lcm$(m,a)$ which is a divisor of $n$, so $(q_1,q_2)\in Y_0(n')$ for some $n'|n$.		
	\end{proof}
	
	\begin{cor}\label{Tatehowmanytori}
		Fix $R<1$. The intersection of the modular curve $Y_0(n)$ with the ball or radius $R$ around $(\infty,\infty)\in \pee^1\times \pee^1$ consists of a union of at most $n^{1/4}$ branches of the form $q_2^a=q_1^b$.
	\end{cor}
	
	\begin{proof}
		The Lemma shows that the number of branches is bounded by the square of the number of divisors of $n$. It follows from the Prime Number Theorem as in (\cite{Apo76}, Corollary to Theorem 13.12), that, for any $\delta>0$, the number of divisors of $n$ is $o(n^\delta)$ as $n\to \infty$. Taking $\delta =1/8$ gives the Corollary. 
	\end{proof}

	\begin{prop}\label{Tatesingletorus}
		Fix $R<1$. There exists a constant $C_1$ (depending on $R$ but independent of $n$) such that for any $n\in\enn$ and any, $a$ dividing $n$, there are at most $C_1\sqrt{n}$ intersections of $C$ with the locus $q_2^a = \zeta q_1^{n/a}$ in the ball of radius $R$ in Tate coordinates.

	\end{prop}

	\begin{proof}
		Without loss of generality, assume $a\leq\frac{n}{a}$. Let $d =$ hcf $(a,\frac{n}{a})$ and write $a=de$.
		
		Write $C$ in Tate coordinates as the vanishing of $F(q_1,q_2) = \sum_{ij} c_{i,j}q_1^iq_2^j$, convergent for $|q_i|<1$. Choose $\pi \in \cp$ with $R<|\pi|<1$ and write $t=|\pi|$. Write $q_1 = \pi x$, so that, after substituting below, there will be convergence for $|x|\leq 1$. Then writing $r = \frac{n}{a^2}$, have $q_2 = \zeta (\pi x)^r$ for some $a^\text{th}$ root of unity $\zeta$.

		Consider the coefficients $c_{i,0}$. If these are all zero, $F(q_1,q_2) = q_2G(q_1,q_2)$ for some $G$, but $C$ is irreducible so $C = \mathbb{V}(q_2)$. This means that $C$ only intersects modular curves at $(\infty,\infty)$ in this residue disc. So may assume there exists $u$ with $c_{u,0}\ne 0$. Let $A = |c_{u,0}||\pi|^u>0$.
		
		\textbf{Case 1:} Suppose first that $\frac{n}{ad}>u$. Consider the expansion
		\[F(\pi x,\zeta(\pi x)^r) = \sum_{ij} c_{i,j}\pi^ix^i\zeta^j\pi^{rj}x^{rj} = \sum_k b_k x^{k/e}\] where $b_k = \pi^{k/e}\sum_{ei+(nj)/(ad)=k} c_{i,j}\zeta^j$. For $k=eu$, only one term contributes, else if $i$ and $j$ were such that $ei+(nj)/(ad) = eu$, would have $e$ dividing $j$ since $e$ is coprime to $n/(ad)$, but then $(nj)/(ad)$ already exceeds $eu$ by the assumption that $n/(ad)>u$. Thus have $b_{eu} = \pi^uc_{u,0}$, so $|b_{eu}|=A$.
		
		Now changing variable to $s = x^{1/e}$ (the decomposition $a=de$ is such that $n/a^2$ has denominator $e$ in lowest terms), the coefficients of the power series $\sum_k b_ks^k$ satisfy $|b_k|\leq t^{k/e}$, where $t=|\pi|$. That is, choosing $I$ such that $t^{I/e}<A$, the number of roots of the power series for $|s|\leq 1$ can number at most $I$. Solving, must choose $I>e\frac{\log{A}}{\log{t}}$. So the number of roots is at most $C_2e \leq C_2a \leq C_2\sqrt{n}$. The roots are recovered as $(q_1,q_2) = (\pi s^e, \zeta\pi^rs^{n/ad})$.
		
		A choice of $\zeta$ was made in this part of the argument, but each choice counts the same roots in $(q_1,q_2)$-coordinates. For instance, looking at the branch $q_2 = q_1^r$ instead (i.e. taking $\zeta=1$), consider the expansion		
		\[F(\pi x, (\pi x^r))=  \sum_{ij} c_{i,j}\pi^ix^i\pi^{rj}x^{rj} = \sum_k d_k z^k\] where $z = x^{1/e}$ and $d_k = \pi^{k/e}\sum_{ei+(nj)/(ad)=k} c_{i,j}$. Making the substitution $z=\zeta^ls$ for $l$ an inverse to $n/ad$ modulo $e$ transforms this equation into the previous one and indicates that the solutions coming from $(q_1,q_2) = (\pi s^e, \zeta\pi^rs^{n/ad})$ are exactly the same as those coming from $(q_1,q_2) = (\pi s^e, \pi^rz^{n/ad})$. Indeed this overcounting can also be observed from the fact that the cover given by the $s$ coordinates is increasing multiplicity (away from 0, at least) by $e$.

		\textbf{Case 2:} It remains to deal with the cases where $n/(ad)\leq u$. Since $a\leq n/a$, $e = a/d = (n/(ad))/(n/a^2)\leq n/(ad)\leq u$, so there are finitely many possibilities for $e$ and $n/(ad)$, both less than $u$. Thus for each of these finitely many choices, need to bound, running across all roots of unity $\zeta$, the number of zeros of $\sum_kb_k(\zeta)x^{k/e}$ where		
		\begin{align}\label{rootsof1}
			b_k(\zeta) = \pi^{k/e}\sum_{ei+(nj)/(ad)=k} c_{i,j}\zeta^j.
		\end{align} Note that infinitely many different $\zeta$ can appear coming from the cases $n=a^2$, where $a$ is arbitrary.
		
		Choose $K$ such that $b_K(1)\ne0$. Then $b_K(\zeta)$ is a non-zero polynomial in $\zeta$, so has at most finitely many zeros, $\zeta_1,...,\zeta_h$, in roots of unity. Now by \cite{Ber02}, there exists $c>0$ such that if $\zeta\ne \zeta_i$ for any $i=1,...,h$ then $|b_K(\zeta)|>c$.

		As before, let $I$ be such that $t^{I/e}<c$. Recall that $e$ is being thought of as fixed as one of the finitely many exceptional $e<u$. This means that for any $\zeta\ne \zeta_i$, there are at most $C_3 = C_3(e,n/(ad)) \colonequals e\frac{\log{c}}{\log{t}}$ many roots with $|q_i|<R$ by Strassman's Theorem. If $D_i(e,n/(ad))$ is the number of roots contributed by the zero $\zeta_i$ of $b_K(\zeta)$, and \[C_4 \colonequals \max_{\{e,n/(ad)\leq u\}}\{C_3(e,n/(ad)),D_1(e,n/(ad)),...,D_k(e,n/ad)\},\] can conclude that there are at most $C_4$ roots of (\ref{rootsof1}) for any of the exceptional choices. Thus the constant $C_2$ from above may be enlarged to a constant $C_1$ to account for these finitely many cases, coming from branches of $Y_0(n)$ where both $e$ and $n/(ad)\leq u$.
	\end{proof}

	\begin{thm}\label{ThmInfty}
		The equidistribution condition is satisfied on all balls of radius $R<1$ around infinity.
	\end{thm}

	\begin{proof}
		Fix $R<1$. As argued at the start of this section, the only balls which can contain intersections are those around $(\infty,\infty)$, thus Tate coordinates centred at $(0,0)$ can be used. Given any curve $C$ in Tate coordinates, Proposition \ref{Tatesingletorus} says that there are at most $C_1\sqrt{n}$ intersections between $C$ and the locus $q_2^a=q_1^b$ with $ab=n$. Corollary \ref{Tatehowmanytori} says that $Y_0(n)$ is comprised of at most $n^{1/4}$ such branches inside the ball of radius $R$ around infinity. Hence 
		\[|C\cap Y_0(n)\cap B_R(\infty,\infty)|\leq C_1n^{3/4},\] whence the limit in (\ref{ordseq}) is seen to be zero as required.
	\end{proof}
	
	\begin{thm}
		Theorem \ref{Main} holds.
	\end{thm}
	
	\begin{proof}
		This is now a matter of combining Lemma \ref{padicconditionproof} (which translated the Theorem into showing some limits were zero) with Theorems \ref{ThmOrd}, \ref{ThmSsing} and \ref{ThmInfty} (which computed these limits in the ordinary, supersingular and multiplicative reduction cases).
	\end{proof}

\printbibliography

\end{document}